\documentclass{article}

\usepackage{amsmath,amssymb,graphicx} %load extra symbols and environments
\usepackage{amssymb}
\usepackage{amsthm}
\usepackage[margin=0.9in]{geometry} %set margins
\newtheorem{theorem}{Theorem}[section]
 \newtheorem{definition}{Definition}[section]
 \newtheorem{lemma}{Lemma}[section]
\newtheorem{remark}{Remark}[section]
 \usepackage{hyperref}
 \usepackage{enumerate} 
 \allowdisplaybreaks
 \usepackage{comment} 
\usepackage{tikz}
 \usepackage{xcolor}
\newcommand{\clb}{\color{black}}
\begin{document}
\nocite{*} % this command forces all references in template.bib to be printed in the bibliography

%\title{Error analysis of finite element approximation of one dimension moving boundary problem for capturing the diffusants penetration of rubber }

\title{Error estimates for semi-discrete finite element approximations for a moving boundary problem  capturing the penetration of diffusants into rubber}
% One possible (more descriptive) suggestion from YW

 \author{Surendra\,Nepal $^{1,*}$,  Yosief\,Wondmagegne$^1$, Adrian\,Muntean$^1$ \\
$^1$ Department of Mathematics and Computer Science, Karlstad University, Sweden\\
*surendra.nepal@kau.se}

\date{\today} % if this is omitted, the current date is used for the title page

\maketitle

%-----------------------------
\noindent
\begin{abstract}
%{\color{red} Working file}\\
%\noindent 
{\clb We consider a moving boundary problem with kinetic condition that describes the diffusion of solvent into rubber and study semi-discrete finite element approximations of the corresponding weak solutions.} We report on both \textit{a priori} and \textit{a posteriori} error estimates for the mass concentration of the diffusants, and respectively, for the {\clb {\em a priori unknown}} position of the moving boundary. Our working techniques include integral and energy-based  estimates for a nonlinear parabolic problem posed in a transformed fixed domain combined with a suitable use of the interpolation-trace inequality to handle the interface terms.
Numerical illustrations of our FEM approximations are within the experimental range and show good agreement with our theoretical investigation. {\clb This work is a preliminary investigation necessary before extending the current moving boundary modeling to account explicitly for the mechanics of hyperelastic rods  to capture a directional swelling of the underlying elastomer.}

\medskip


%{\bf ADD the MSC 2020 classification!}
\end{abstract}
%-----------------------------

\section{Introduction}
%1) SHORT description of the setting concerning the penetration of rubber by diffusants. Identify also some rubber engineering papers that we have not yet cited but talk about the same basic scenario. \\
%2) The aim of the paper. (what scientific questions we ask here?)\\
%3) What is the new results that we get in this research work?\\

%Rubber based materials are used in many  industrial and household purposes. Understanding the penetration of diffusants into rubber based material is an important task.  Moving boundary approach is one of the approaches  to describe such phenomena where one has to determine the boundary of the domain as a part of the solution. 
%We describe the setting of the problem as follows:  

Sharp interfaces moving in an {\em a priori} unknown way inside materials play a key role in a number of study cases in science and technology, including in the forecast of the durability of cementitious-based materials (cf. e.g.  \cite{chainais2018convergence, muntean2008error, muntean2009moving, zurek2019numerical}), large-time behavior of chemical species from the environment slowly penetrating by diffusion and swelling rubber-based  materials (cf. e.g. \cite{aiki2020macro, kumazaki2021free, nepal2021moving}), to controlling phase transitions like melting and freezing or solid-solid changes in concrete (cf. e.g. \cite{alexiades1992mathematical,piqueras2018numerical, piqueras2020solving}), to mention but a few. Due to the inherent non-linearity of such moving boundary problems, analytical representations of solutions are often either unavailable or not computable. Hence, one has to rely on direct computational approaches to get insight for instance in the behavior of large times of such moving sharp interfaces, as this usually defines the lifetime of the material under investigation.  

In the framework of this paper, we study a semi-discrete finite element approximation of weak solutions to a one dimensional moving boundary problem that models the diffusion of solvent into rubber (see Section \ref{statement}). This is a follow-up study  of our recent work \cite{nepal2021moving}, where we proposed a finite element  approximation of solutions to a moving boundary problem which we used to recover experimental data. Now, we explore the quality of our approximation scheme. {\clb Specifically}, we report on both \textit{a priori} and \textit{a posteriori} error estimates for the mass concentration of the diffusants, and respectively, for the position of the moving boundary. Our working techniques include integral and energy-based  estimates for the corresponding nonlinear parabolic problem posed in a transformed fixed domain, combined with a suitable use of the interpolation-trace inequality to handle the interface terms. At the technical level, we were very much inspired by the references: \cite{caboussat2005analysis, heywood1982finite,nitsche1978finite, pani1991finite}, and \cite{muntean2008error}.
%inefficient  even in one-dimensional setting. Numerical solutions therefore have become the main tools for solving moving boundary problems. %At the numerical point of view, there exists different numerical methods to determine the approximation position of the penetration depth. In \cite{piqueras2020solving} the authors use the front fixing technique to solve two phase Stefan problem with phase formation and depletion by using finite difference method. 
It is worth noting that similar work has been done in related contexts. For instance, in \cite{chainais2018convergence}, the authors show the convergence of a numerical scheme obtained by combining an Euler discretization in time
with  a Scharfetter-Gummel discretization in space for a concrete carbonation model with moving boundary reformulated for  a fixed space domain. In \cite{zurek2019numerical}, A. Zurek studies the long time regime of the moving interface driving the concrete carbonation reaction model by tailoring an implicit in time and
finite volume in space scheme. He proves that the approximate free boundary increases in time with $\sqrt{t}$-law as theoretically predicted in \cite{aiki2011free}.  In \cite{mackenzie2000numerical}, one develops an adaptive moving mesh method for the numerical solution of
an enthalpy formulation of a class of heat-conduction problems with  phase change. The main aim of \cite{javierre2006comparison} is to provide a  comparison of several numerical methods including displacing level sets, moving
grids, and diffusing phase fields to address two well-known Stefan problems arising as best formulations for phase transformations like  melting of a pure phase and diffusional solid-state phase changes in binary systems.

{\clb To handle our problem, we decided to use the finite element method as this fits best to the regularity of the (weak) solutions to our moving boundary problem. Mind though that other discretization methods are likely to be applicable as well. As our work is purely in 1D and no expensive computations are expected, and as, on top of this, we wish to rely on open source facilities, we chose {Python} for the implementation work.}

{\clb We present here a preliminary investigation of this class of problems. This is necessary before extending the current moving boundary modeling to account explicitly for the mechanics of hyperelastic rods  to capture a directional swelling of the underlying elastomer. In this spirit, a natural next step  would be to perform the numerical analysis of a two-scale finite element approximation of the setup described in \cite{aiki2020macro}.}

The outline of this study is as follows: We formulate our moving boundary problem in Section \ref{statement}. The discussion of the setting of the model equations  is based on \cite{nepal2021moving}. We collect in Section \ref{preliminaries} our basic assumptions on parameters {\clb and model components}, as well as notations and existing preliminary results. Section \ref{transformation} contains the fixed domain transformation of our problem and the definition of our concept of weak solutions which is then the subject of error approximation estimates investigated here. {\clb Benefiting of the mathematical analysis done for our problem in \cite{NHM,kumazaki2020global}, we are able to} prove  the global existence  of weak solutions to the semi-discrete problem and obtain the needed uniform boundedness results to produce convergent numerical schemes.  As main result, we obtain \textit{a priori} and \textit{a posteriori} error estimates as listed in Section \ref{mainresults}. A couple of numerical experiments are  discussed in Section \ref{numerical}.  Essentially, they support  numerically the available experimental results. %They  are in good agreement with our theoretical investigation. 
Finally, a brief conclusion of this work is outlined in Section \ref{conclusion}.

\section{Model equations}\label{statement}
%In this paper, we estimate the \textit{a priori} and \textit{a posteriori} error based on finite element Galerkin method for a moving boundary problem that models a penetration of the solvent into a rubber.\\
%At time $t = 0$, we consider a piece of a rubber filling a dimensional rigid rod denoted by $\Omega$ of vertical length $L>0$. This is placed in contact with a diffusant reservoir. When the diffusant concentration at the face of the rubber at $x = 0$ exceeds some threshold, the diffusant moves into the rubber creating a sharp interface that separates the rubber $\Omega$ into two parts, the diffusant free region and diffusant-penetrated region. Our interest of region is the diffusant-penetrated region where  the diffusant's flux is assumed to satisfy Fick's law. In addition, in order to fully specify the action of the free boundary, an additional condition must be imposed in the form of a mass conservation condition.  \\
%As the diffusants-free zone  does not contain any diffusant for each time $t$,  the actual  problem is to find the  diffusant concentration profile inside the diffusant-penetrated region and the location of the moving interface.  This setting is called  one-phase moving boundary problem; see e.g. \cite{gupta2017classical} concerning modeling with moving interfaces. 
We consider a thin slab of a dense rubber, denoted by $\Omega$ of vertical length $L>0$, placed in contact with a diffusant reservoir. When the diffusant concentration at the bottom face of the rubber exceeds some threshold, the diffusant moves into the rubber creating a sharp interface that separates the rubber $\Omega$ into two parts, the diffusant free region and diffusant-penetrated region. Our {\clb region of interest} is the diffusant-penetrated part where  the diffusant's flux is assumed to satisfy Fick's law. The actual  problem is to find the  diffusant concentration profile inside the diffusant-penetrated region and the location of the moving interface {\clb separating the penetrated from the not-yet penetrated region}. Such a  setting is referred to as a   one-phase moving boundary problem. {\clb Formulations as a two-phase boundary problem are  possible as well, but are currently not in our focus;} see e.g. \cite{gupta2017classical} for a {\clb nicely written } textbook regarding modeling with moving interfaces. 

In this work, the modeling domain is 
%a 1D slab as 
the one–dimensional slab
shown in Figure \ref{Fig:2}, which is the longitudinal line where  $0<s(0)\leq s(t)\leq L$.
\begin{figure}[ht]
	\begin{center}
		\begin{tikzpicture}[scale=1, every node/.style={scale=1}]
		\draw[|-|, color=red] (-1,-4) -- (5,-4);
		\draw[|-|, color=red] (3,-4) -- (5,-4);
		\draw (0, -4.1) -- (0, -3.9);
		%	\draw (0, -4.1) -- (0, -3.9);
		\node at (-1,-4.4) {$ 0$};
		\node at (0,-4.4) {$s(0)$};
		\node at (3,-4.4) {$s(T_f)$};
		\node at (5,-4.4) {$L$};
		\draw[thick, <-]  (-0.5, -4) -- (-0.5, -5)node[anchor=north] {Zone inside rubber with penetration of diffusants at $t=0$} ;
		\draw[thick, <-]  (1.5, -4) -- (1.5, -3)node[anchor=south] {Diffusant-free rubber at $t=0$} ;
		\end{tikzpicture}
		\caption{Sketch of one dimensional geometry -- a macroscopic thin slab made of rubber.}
		\label{Fig:2}
	\end{center}
\end{figure}
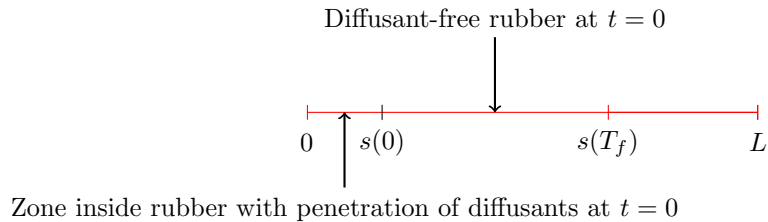
For a fixed observation time $T_f\in (0, \infty)$, the interval $[0, T_f]$ is the time span of the process we are considering.  Let  $x\in [0, s(t)]$ and $t \in [0, T_f]$ denote the space and respectively time variable,  and let $m(t, x)$  be the concentration of diffusant placed in position $x$ at time  $t$.
 The diffusants concentration  $m(t, x)$  acts in the region $Q_s(T_{f})$ defined by $$ Q_s(T_{f}):= \{ (t, x) | t \in (0, T_{f}) \; \text{and}\; x \in (0, s(t))\}.$$ 
 The problem reads: Find 
 %the diffusant 
 $m(t, x)$ and the position of the moving interface $x = s(t)$ for $t\in(0, T_f)$ such that the couple $(m(t, x), s(t))$ satisfies the following % equations
\begin{align}
\label{a11}&\displaystyle \frac{\partial m}{\partial t} -D \frac{\partial^2 m}{\partial x^2} = 0\;\;\; \ \text{in}\;\;\; Q_s(T_{f}),\\
\label{a12}
&-D \frac{\partial m}{\partial x}(t, 0) = \beta(b(t) -\text{H}m(t, 0))  \;\;\; \text{for}\;\;  t\in(0, T_{f}),\\
\label{a13}&-D \frac{\partial m}{\partial x}(t, s(t))  =s^{\prime}(t)m(t, s(t))  \;\;\; \text{for}\;\; t\in(0, T_{f}),\\
\label{a14}&s^{\prime}(t) = a_0 (m(t, s(t)) - \sigma(s(t)) \;\;\;\;\text{for } \;\;\; t \in (0,T_{f}),\\
\label{a15}&m(0, x) = m_0(x) \;\;\;\text{for}\;\;\; x \in [0, s(0)],\\
\label{a16} & s(0) = s_0>0\; \text{with}\;\; 0<  s_0< s(t) < L, 
\end{align}
where  $a_0>0$ is a  kinetic coefficient, $\beta$ is a positive constant,  $D>0$ is a diffusion constant, $\text{H}>0$ is the Henry's constant,  $\sigma$ is a function on $\mathbb{R}$, $b$ is a given boundary function on $[0, T]$, and
$s_0>0$ is the initial position of the free boundary and $m_0$ is the initial concentration of the diffusant.

The boundary condition \eqref{a13} describes the mass conservation of diffusant concentration at the moving boundary. It indicates that the diffusion mechanism is responsible for pushing the interface.  In particular \eqref{a14} points out that the mechanical behaviour (here it is about the swelling of the rubber) also contributes to the motion of the moving penetration front. The explanation of the model equations and the physical meaning of the parameters are given in \cite{nepal2021moving}.
%It is worth mentioning that when the term $\sigma(s(t))$ vanishes, the model  equations  form the so-called one-phase  Stefan problem with  kinetic  condition. This approach was originally used to model the motion of the ice-water interfaces (compare e.g. \cite{alexiades1992mathematical}) as well as the motion of sharp carbonation reaction interfaces in concrete (see  \cite{muntean2009moving}).

\section{Notations, assumptions and preliminaries} \label{preliminaries}
In this section, we list our basic assumptions on the data, notations as well as approximation properties of functions that are required  for the error analysis discussed in the next sections.
\subsection{Function spaces and elementary inequalities}
Let $u, v : \Omega \to \mathbb{R}$ denote two generic functions. Let $W^{r, p}(\Omega)$ be the  Sobolev space on domain $\Omega$ for $1\leq p \leq \infty$ and $r\geq 0$.  For $r = 0$, we simply write $L^p(\Omega)$ in place of  $W^{0, p}(\Omega)$ with the norm $\|\cdot\|_{L^p(\Omega)}$  defined as follows:
\begin{equation*}
\|u\|_{L^p(\Omega)} := \begin{cases}  \left( \displaystyle \int_{\Omega} |u(x)|^p dx  \right)^{\frac{1}{p}} \;\;&\text{for} \;\; 1\leq p < \infty,\\
 \text{ess sup} \{ |u(x)|: x\in \Omega\} \;\;  &\text{for} \;\;  p = \infty,
\end{cases}
\end{equation*}
For $p=2$ and $r\geq1$, we write $H^r(\Omega)$ in place of  $W^{r, 2}(\Omega)$ with the norm $\|\cdot\|_{H^{r}(\Omega)}$ defined by
\begin{align}\label{hnorm}
\|u\|_{H^r(\Omega)} =  \left(\sum_{|\alpha|\leq r} \int_{\Omega} |\partial^{\alpha} u|^2 dx \right)^{\frac{1}{2}}.
\end{align}
In \eqref{hnorm} $\partial^{\alpha} u$ denotes the $\alpha$'th derivative of $u$ in the weak sense.
Furthermore, for $L^2(\Omega)$ and $H^r(\Omega)$ we have the following inner products.
\begin{align*}
    &(u, v)_{L^2(\Omega)} :=  \int_{\Omega} u(x) v(x) dx,\\
    &(u, v)_{H^r(\Omega)} := \sum_{|\alpha|\leq r}(\partial^r u, \partial^r v)_{L^2(\Omega)}.
\end{align*}
Let $X$ be a Banach space with norm $\| \cdot \|_{X}$ and  $v:[0,T] \rightarrow X$ be a function. Correspondingly, $L^p(0, T, X)$ is a Bochner space endowed with the norms
\begin{align*}
 \| v\|_{L^p(0, T, X)}:= \begin{cases} \left( \displaystyle \int_0^T \| v(\tau)\|_{X}^p d\tau\right)^{\frac{1}{p}} \;\;\;\;\;&\text{for}\;\;\; 1\leq p < \infty,\\
 \displaystyle \sup_{0 \leq \tau \leq T } \|v(\tau)\|_{X}\;\;\; &\text{for}\;\; p = \infty.
\end{cases}
\end{align*}
More information on Sobolev and Bochner spaces with their various norms and inner products can be found for instance in  \cite{adams2003sobolev, kufner1977function}. 
For the convenience of writing, we denote  $u(t, 0)$ and $u(t, 1)$ by $u(0)$ and $u(1)$, respectively. We also use the prime $(^\prime)$ to point out the derivative with respect to time variable, and  $\|\cdot\|$ and $(\cdot,\cdot)$ for the norm and, respectively, inner product in $L^2(\Omega)$. { \clb Furthermore, $\|\cdot\|_\infty$ refers to the norm of $L^\infty(\Omega)$.}\\
We list a few elementary inequalities that we  frequently use in this work.
\begin{enumerate}[(i)]
\item Young's inequality:
\begin{align}
\label{a3}ab \leq \xi a^p + c_{\xi} b^q,
\end{align}
where $a, b \in \mathbb{R}_{+},\; \xi >0,\; c_{\xi}:= \displaystyle \frac{1}{q} \frac{1}{\sqrt[p]{(\xi p)^q}} >0, \; \displaystyle \frac{1}{p} + \frac{1}{q} = 1$ and $p \in (1, \infty).$
\item Interpolation inequality: { \clb For all $u \in H^1(0, 1)$,
%{\clb In 1D, the relation %$H^1(0, 1) \subset L^{\infty}(0,1) \subset L^2(0,1)$ holds true.}{ \clb let $\xi$ and $c_{\xi}$ be as in \eqref{a3}.}
 there exists a constant $\hat{c}>0$ depending on $\theta \in [\frac{1}{2}, 1)$ such that 
%\begin{align*}
%\label{a33}
%\| u\|_{\infty} \leq \hat{c} \|u\|^{\theta} \left\Vert u\right\Vert_{H^1(0,1)}^{1 - \theta} \leq \hat{c}\left(\xi \left\Vert \frac{\partial u}{\partial y}\right\Vert + ( \xi + c_{\xi}) \|u\| \right),
%\end{align*} 
\begin{align}
\label{a33}
\| u\|_{\infty} \leq \hat{c} \|u\|^{\theta} \left\Vert u\right\Vert_{H^1(0,1)}^{1 - \theta}.
%\leq \hat{c}\left(\xi \left\Vert \frac{\partial u}{\partial y}\right\Vert + ( \xi + c_{\xi}) \|u\| \right),
\end{align}
For  $\theta = 1/2$, one gets 
\begin{align*}
\| u\|_{\infty}^2\leq \hat{c}\left(\xi \left\Vert \frac{\partial u}{\partial y}\right\Vert^2 + ( \xi + c_{\xi}) \|u\|^2 \right),
\end{align*} 
where $\xi$ and $c_{\xi}$ are as in \eqref{a3}. See details in \cite{zeidlernonlinear} p. 285 (example 21.62).}
%\item The inequality 
%\begin{align*}
%|a + b| ^p   {\clb \leq}
%\begin{cases}
%|a|^p + |b|^p,&\;\;\;\;\;\; \text{for } p \in  (0, 1)\\
%(1 + \xi)^{p-1} |a|^p + \left(1 + \displaystyle \frac{1}{\xi} \right)^{p-1} |b|^p   &\;\;\;\;\;\; \text{for}\;\; p\in [1, \infty)
%\end{cases}
%\end{align*}
%holds for arbitrary $a, b \in \mathbb{R}$ and $\xi>0.$
\end{enumerate}

\subsection{Assumptions on parameters}
Throughout this paper, we assume the following restrictions on the parameters. 
\begin{enumerate}[({A}1)]
	\item \label{A1} 	$a_0,\;\text{H},\; D,\; s_0, \; T_{f}$ are positive constants.
	\item  	\label{A2}$b \in W^{1, 2}(0, T_{f})$ with $0< b_{*} \leq b \leq b^{*}$ on $(0, T_{f})$, where $b_{*}$ and $b^{*}$ are positive constants.
	\item 	\label{A3} $\beta \in C^1(\mathbb{R}) \cap W^{1, \infty} (\mathbb{R})$ such that $\beta = 0$ on $(\infty, 0]$, and there exists $r_{\beta}>0$ such that $\beta^{\prime}>0$ on $(0, r_{\beta})$ and $\beta = k_0$ on $[r_{\beta}, +\infty)$, where $k_0>0$.
	\item 	\label{A4} $\sigma \in C^1 (\mathbb{R}) \cap W^{1, \infty} (\mathbb{R})$ such that $\sigma = 0$ on ${\clb (-\infty, 0 )}$, and {\clb there exists}  $r_{\sigma}$ such that $\sigma^{\prime}>0$ on $(0, r_{\sigma})$ and $\sigma = c_0$ on $[r_{\sigma}, +\infty),$ where $c_0$ {\clb satisfies} 
\begin{align}
 0 <c_0 < \min\{ 2 \sigma(0), b^*\text{H}^{-1}\}.
\end{align}
\item 	\label{A5}$0<s_0< r_{\sigma}$ and $m_0 \in H^1(0, s_0)$ such that $\sigma(0) \leq u_0 \leq b^*\text{H}^{-1}$ on $[0, s_0].$
\end{enumerate}
The assumptions (A\ref{A1})--(A\ref{A5}) are adopted from \cite{kumazaki2020global}, where the authors have proved the global solvability of the problem and continuous dependence estimates of the solution  with respect to the initial data. 

\subsection{Basic facts from approximation theory}
Let $N \in \mathbb{N}$ be 
%a given integer. 
given.
We set $0 =y_0 < y_1 < \cdots < y_{N-1} = 1$ as discretization points in the interval $[0, 1]$. We set $k_i := y_{i+1} - y_{i}$ {\clb for $i \in \{0,1,\cdots, N-2\}$} and {\clb $k := \displaystyle \max \{k_i: \ i \in \{0,1,\cdots, N-2\}\}$}. We introduce the space 
\begin{align} \label{finite}
V_k := \{ \nu \in C[0, 1]: \nu |_{[y_j, y_{j+1}]}  \in \mathbb{P}_1\}, 
\end{align} 
where $\mathbb{P}_1$ represents the set of 
%the 
polynomials of degree one.
%If $u_{0,k} $ is the  Lagrange interpolant of $u_0 \in H^1(0, 1)$ in $V_k$ (see \cite{larsson2008partial} for the definition of Lagrange interpolation), then we have $\| u_{0, k}\|_{H^1(0,1)} \leq \| u_0\|_{H^1(0, 1)}$. Furthermore, if $u_0 \in H^2(0, 1)$, then the classical interpolation result gives
%\begin{align*}
%\| u_0-u_{0,k}\|_{L^2(0,1)} \leq c k^2 \| u_0\|_{H^2(0, 1)},
%\end{align*}
%where $c$ is a positive constant independent of $k$.\\
{\clb Let $\{\phi_i\}_{i=0}^{N-1}$ be the set of basis functions for the space $V_k$ defined by
\begin{equation*}
\phi_i(y) =
\begin{cases}
0 &\;\;\;\;\;\; \text{if}\;\;\;\;y< y_{i-1}\\
\displaystyle \frac{ y - y_{i-1}}{k_{i-1}}  & \;\;\;\;\;\;\text{if}\;\;\;y_{i-1} \leq  y < y_i\\
\displaystyle \frac{ y_{i+1} - y}{k_{i}}  &\;\;\;\;\;\; \text{if}\;\;\;y_{i}\leq y < y_{i+1}\\
0 &\;\;\;\;\;\; \text{if}\;\;\;\; y_{i+1}\leq y.\\
\end{cases}
\end{equation*}
}
We define the interpolation operator $I_k:C[0, 1] \rightarrow V_k$ by
\begin{align*}
(I_k u)(y):= \sum_{i=0}^{N-1} u(y_i, t)\phi_i(y).
\end{align*}
{\clb Here the function $I_k u$ is called the Lagrange interpolant of $u$ of degree 1; for more} details see e.g. \cite{larsson2008partial}. 
%We denote by  $\mathcal{R}_k: H^1(0,1)\to V_N$ the Ritz projection operator. The projection $\mathcal{R}_k w$ of $w \in H^1(0,1) $ are defined by  $(\nabla (w - \mathcal{R}_k w), \nabla \psi) = 0$ for all $\psi \in V_N$. 
\begin{lemma} \label{lemma1}
	Take $\theta \in [\frac{1}{2}, 1)$ and $\psi\in H^2(0,1)$.  Then there exist  strictly positive constants $\gamma_1,\; \gamma_2$ and $\gamma_3$ such that the Lagrange interpolant $I_k \psi$ of $\psi$ satisfies the following estimates:
	\begin{enumerate}[{\rm(i)}]
	\item  	\label{l1} $\|\psi - I_k\psi \|\leq \gamma_1 k^2 \|\psi \|_{H^2(0,1)}$ 
		\item  \label{l2} $\left\Vert\displaystyle\frac{\partial}{\partial y}  (\psi- I_k\psi) \right\Vert\leq \gamma_2 k \|\psi \|_{H^2(0,1)}$ 
	\item 	\label{l3}	$|\psi(0) - I_k\psi(0)|\leq {\clb \hat{c} \left(  \gamma_1 k^2 + \gamma_3 k^{1+\theta} \right) \|\psi \|_{H^2(0,1)}}$ 
				\item \label{l4}	$|\psi(1) - I_k \psi(1)|\leq \hat{c} \left(  \gamma_1 k^2 + \gamma_3 k^{1+\theta} \right) \|\psi \|_{H^2(0,1)}$ 
		\end{enumerate}
	\end{lemma}
\begin{proof}
The inequalities (\ref{l1}) and (\ref{l2}) are standard results. For details on their proof, see for instance page 61 in \cite{larsson2008partial} and page 3 in \cite{thomee2007galerkin}. To show (\ref{l3}), we use the interpolation inequality \eqref{a33} together with (\ref{l1}) and (\ref{l2}) to obtain
{\clb	\begin{align*}
	|\psi(0) - I_k\psi(0)|&\leq \hat{c} \|\psi - I_k\psi \|^{\theta}_{L^2(0,1)} \left\Vert  \psi- I_k\psi \right\Vert^{1 - \theta}_{H^1(0,1)}\\
	&\leq \hat{c} \|\psi - I_k\psi \|^{\theta}_{L^2(0,1)} \left(\|\psi - I_k\psi \|^{1- \theta} + \left\Vert\displaystyle\frac{\partial}{\partial y}  (\psi- I_k\psi)\right\Vert^{1-\theta} \right)\\
	& \leq \hat{c} \left(  \gamma_1 k^2 + \gamma_1^{\theta} \gamma_2^{1-\theta} k^{1+\theta} \right) \|\psi \|_{H^2(0,1)}.
	\end{align*}}
	Taking $\gamma_3 :=  \gamma_1^{\theta} \gamma_2^{1-\theta} $ gives the estimate  $({\rm\ref{l3}})$. A similar argument applied to $\psi(1)$ gives $({\rm\ref{l4}})$. 
	\end{proof}
	
\section{Fixed-domain transformation and definition of weak solutions}
\label{transformation}
Firstly, we perform the non-dimensionalization of the model equations \eqref{a11}--\eqref{a16}. We then transform the non-dimensional model equations from the a priori unknown non-cylinderical domain into the cylinderical domain $Q(T) := \{ (\tau, y) |\; \tau \in (0, T) \; \text{and}\; y \in (0, 1)\}$ by using the Landau transformation $y = x/s(t)$, see for instance \cite{landau1950heat}. For more details on non-dimensionalization and transformation, we refer the reader to \cite{nepal2021moving} where the preliminary steps are done. In dimensionless form, the transformed  problem {\clb reads as follows:}
\begin{align}
\label{aa17} 
&\displaystyle\frac{\partial{u}}{\partial{\tau}} - y\frac{h^{\prime}(\tau)}{h(\tau)}  \frac{\partial{u}}{\partial{y}}- \frac{1}{(h(\tau))^2}\frac{\partial^2{u}}{\partial{y^2}} = 0
\;\;\; \ \text{in}\;\;\;Q(T),\\
\label{aa18} &- \frac{1}{h(\tau)} \frac{\partial u}{\partial y}(\tau, 0) = \text{\rm{Bi}}\left(\frac{b(\tau)}{m_{0}}- \text{H}u(\tau, 0)\right)\;\;\; \text{for}\;\; \tau\in(0, T), \\
\label{aa19}&- \frac{1}{h(\tau)} \frac{\partial u}{\partial y}(\tau, 1) =  h^{\prime}(\tau)u(\tau, 1) \;\;\; \text{for}\;\; \tau\in(0, T), \\
&\label{aa20} h^{\prime}(\tau) =A_0\left(u(\tau, 1)-\frac{\sigma(h(\tau))}{m_{0}}\right)\;\;\; \text{for}\;\; \tau\in(0, T)\\
\label{aa21} &u(0, y)  = u_0(y) \;\;\; \text{for}\;\; y\in[0, 1], \\
\label{aa22}&h(0) = h_0.
\end{align} 
We refer to the system \eqref{aa17}--\eqref{aa22} posed in the cylinderical domain  $Q(T)$  as problem $(P)$. 
%The weak formulation of the model is to find  the couple $(u, h)$ satisfying the equations 
\begin{remark}
We refer the reader to \cite{nepal2021moving} for the definition of dimensionless quantities $u,\;h,\; \tau,\;y,\;T,\; {\rm Bi},\; A_0$. Here we only mention that {\rm Bi} is the mass transfer Biot number and $A_0$ is the Thiele modulus.
	\end{remark}
\begin{definition}
	\label{D1}
	{\rm (Weak Solution to ($P$))}. We call the couple $(u, h)$ a weak solution to problem {\rm ($P$)} on $S_T:=(0, T)$ if and only if  
	\begin{align*}  &h\in W^{1, \infty}(S_{T}) \;\;\text{with}\;\;h_0 < h(T) \leq L,\\
	%\label{a1}\label{a4}
	&u \in W^{1,2}(Q(T)) \cap L^{\infty}(S_{T}, H^1(0, 1)) \cap L^2(S_{T}, H^2(0, 1)),
	\end{align*} 
	such that for all $\tau \in S_{T}$ the following relations hold 
\begin{align}
\nonumber\displaystyle \left(\frac{\partial u}{\partial \tau}, \varphi \right)&-  \frac{h^{\prime}(\tau)}{h(\tau)}\left( y\frac{\partial u}{\partial y},  \varphi \right) +  \frac{1}{(h(\tau))^2}\left(\frac{\partial{u}}{\partial{y}}, \frac{\partial \varphi }{\partial y}  \right)\\
\label{a17}& - \frac{1}{h(\tau)} {\rm Bi}\left(\frac{b(\tau)}{m_{0}}- {\rm H}u(\tau, 0)\right)\varphi(0) +\frac{h^{\prime}(\tau)}{h(\tau)}u(\tau, 1)\varphi (1) = 0 \;\;\text{for all} \; \varphi \in H^1(0,1),\\
& \label{a18} h^{\prime}(\tau) =A_0\left(u(\tau, 1)-\frac{\sigma(h(\tau))}{m_{0}}\right),\\
\label{u}&u(0, y) = u_0(y)\;\;  for \;\; y \in[0, 1],  \\
\label{h}&h(0) = h_0.
\end{align}
\end{definition}

\begin{theorem}
If {\rm (A\ref{A1})--(A\ref{A5})} hold, then problem $(P)$ has a unique solution $(u, h)$ on $S_T$ in the sense of Definition \ref{D1}.
\end{theorem}
\begin{proof} We refer the reader to Theorem 2.4 in \cite{NHM} for a statement of the local existence of weak solutions to problem $(P)$ and to Theorem 3.3 and Theorem 3.4 in \cite{kumazaki2020global} for a way to ensure the  global existence and continuous dependence with respect to initial data.\end{proof}
 We now define the finite element Galerkin approximation to \eqref{a17}--\eqref{h} on the finite dimensional subspace $V_k$.  The semi-discrete approximation $u_k$ and $h_k$ of $u$ and $h$ is now defined to be the mapping
$u_k: [0, T]\rightarrow V_k$ and  $h_k: [0, T]\rightarrow \mathbb{R}_+$ such that \eqref{a19}--\eqref{a22} holds.  We denote the semi-discrete form \eqref{a19}--\eqref{a22} of  problem $(P)$ by $(P_d)$.
\begin{definition}
	\label{D2}
	{\rm (Weak Solution to $(P_d)$)}. We call the couple $(u_k, h_k)$ a weak solution to problem {\rm ($P_d$)} if and only if there is a $S_{{T}}:= (0, {T})$  (for some ${T}>0$) such that 
	\begin{align*}  & h_k\in W^{1, \infty}(S_{{T}}) \;\; \text{with}\;\; h_0 < h_k({T}) \leq L\\
	 & {\clb u_k \in H^1(S_{{T}}, V_k)\cap L^{2}(S_{{T}}, H^1(0, 1)) \cap L^{\infty}(S_{{T}}, L^2(0, 1))}
	\end{align*} 
	and for all $\tau \in S_{{T}}$ it holds
\begin{align}
	\nonumber\displaystyle \left( \frac{\partial u_k}{\partial \tau},  \varphi_k \right)&-  \frac{h_k^{\prime}(\tau)}{h_k(\tau)}\left( y\frac{\partial u_k}{\partial y}, \varphi_k \right) +  \frac{1}{(h_k(\tau))^2}\left(\frac{\partial{u_k}}{\partial{y}}, \frac{\partial\varphi_k }{\partial y} \right)\\
	\label{a19}&- \frac{1}{h_k(\tau)} {\rm Bi}\left(\frac{b(\tau)}{m_{0}}- {\rm H}u_k(\tau, 0)\right) \varphi_k(0) +\frac{h_k^{\prime}(\tau)}{h_k(\tau)}u_k(\tau, 1)\varphi_k (1) = 0 \;\; \text{for all}\;\; \varphi_k \in V_k, \\
	\label{a20} & h_k^{\prime}(\tau) = A_0\left(u_k(\tau, 1)-\frac{\sigma(h_k(\tau))}{m_0}\right),\\
	\label{a21}& u_k(0) = u_{0,k}(y) \;\; {\rm for}\;\; y \in[0, 1],\\
	\label{a22}& h_k(0) = h_0.
\end{align}
\end{definition}
%We now define finite element approximation for the weak solution of the problem \eqref{aa17}-\eqref{aa22}.
\begin{lemma}\label{max} { \clb
Assume {\rm (A\ref{A1})--(A\ref{A5})} hold.  Then there exist a time  $\hat{T}\in (0, T]$ and  positive constants $L, M_1, M_2$ (not depending of $k$)  such that for a.e. $\tau \in (0, \bar{T})$ the following inequalities hold true for the pair $(u_k,h_k)$ arising in Definition \ref{D2}:
\begin{enumerate} [{\rm(i)}]
    \item[(i)] $0< h_0 \leq h_k(\tau) \leq L$
    \item[(ii)] $0< u_k(\tau, y) < M_1 $
    \item[(iii)] $|h_k^{\prime}(\tau)|\leq M_2$.
\end{enumerate}
}
\end{lemma}
\begin{proof} {\clb (i) is built in the concept of weak solution detailed in Definition \ref{D2}. It does not require a proof. We added it here simply to stress the importance of the fact the we work exclusively in a bounded moving domain. (ii) is the main statement here. This holds true as a consequence of the fact that the space continuous version of the statement (i.e. $0< u(\tau, y) < M_1 $) holds true; we rely on the arguments of the proof of  Theorem 3.1 in \cite{kumazaki2020global}, combined with the fact that the treated geometry is one dimensional. Hence,  $\hat{T}>0$ is possibly small, which is sufficient for deriving our next results. Note though that a discrete version of the Stampacchia trick, worked out with details in \cite{jungel2005discrete}, can potentially be applied here as well in order to replace the local time $\hat{T}$ with a maximal time.  Alternative arguments employing the structure of the problem as in \cite{Walter} or based on linear simplicial finite elements as in \cite{Korotov} can also be used in principle. (iii) is a direct consequence of (i) and (ii) combined with \eqref{a20}.}

\end{proof}
\begin{theorem}
	Let  the hypothesis of Lemma \ref{max} be  fulfilled. Then  it exists a unique solution 
	\begin{align*}
	{\clb  (u_k, h_k) \in H^1(S_{\hat{T}}, V_k)\cap L^{2}(S_{\hat{T}}, H^1(0, 1)) \cap L^{\infty}(S_{\hat{T}}, L^2(0, 1)) \times W^{1, \infty}(S_{\hat{T}})}
	\end{align*}
	in the sense of Definition \ref{D2}. Furthermore, there exists a constant $\tilde{c}>0$ (independent of $k$) such that
	\begin{align}
\label{energy}	{\clb \max_{0\leq \tau\leq \hat{T}}} \| u_k\|_{L^2(0, 1)} ^2 + \int_0^{\hat{T}} \left\Vert \frac{\partial u_k}{\partial y}\right\Vert^2_{L^2(0,1)}d \tau \leq \tilde{c}.
\end{align}
\end{theorem}
\begin{proof}% Let $\{\phi_j\}_{j\in \mathbb{N}}$ be an orthogonal basis of $H^1(\Omega)$ as well as an orthonormal basis of $L^2(\Omega)$. 
%Let $\left\{\phi_1, \phi_2, \phi_3,\cdots, \phi_N \right\}$ be  a set of basis for the subspace $V_k \subset H^1(\Omega)$. 
%We  define the projection  operator $P^N$  on finite dimensional subspace $V_N$ associated to the basis $\{\phi_j\}_{j\in \mathbb{N}}$ by
%\begin{align}
 % (P^N\varphi)(y) = \sum_{j=1}^N \alpha_j \phi_j(y)   
%\end{align}
%for any $\varphi(y)\in V_N$.
%\begin{align} \varphi(y) = \sum_{j=1}^\infty \alpha_j \phi_j(y). 
%\end{align} 
%Then the finite element approximation $u_k \in V_k$ of order $N\in \mathbb{N}$ for the function $u$ on the finite dimension subspace $V_k$ is given by 

{\clb Let $V_k$ be the finite dimensional subspace defined in \eqref{finite} constructed based on the span of the hat functions $\{\phi_j\}, j \in \{0, 1,\cdots, N-1 \}$. Let $\alpha_j:(0, \hat{T}) \rightarrow \mathbb{R}$ denote the Galerkin projection coefficient for $j$th degree of freedom.   Then the finite-dimensional Galerkin approximation of the function $u$ is defined by}
\begin{align*}
   u_k (\tau, y) := \sum_{j=0}^{N-1} \alpha_j(\tau) \phi_j(y), 
\end{align*}
where the coefficients $\alpha_j(\tau),\; j \in \{0, 1,\dots N-1\}$ are determined by the following relations:
\begin{align}
\nonumber\displaystyle \left( \frac{\partial u_k}{\partial \tau},  \varphi_k \right)-  \frac{h_k^{\prime}(\tau)}{h_k(\tau)}\left( y\frac{\partial u_k}{\partial y}, \varphi_k \right) &+  \frac{1}{(h_k(\tau))^2}\left(\frac{\partial{u_k}}{\partial{y}}, \frac{\partial\varphi_k}{\partial y} \right)\\
\label{a5}&- \frac{1}{h_k(\tau)} {\rm Bi}\left(\frac{b(\tau)}{m_{0}}- \text{H}u_k(\tau, 0)\right) \varphi_k(0) +\frac{h_k^{\prime}(\tau)}{h_k(\tau)}u_k(\tau, 1)\varphi_k(1) = 0,\\
 \label{a6}&\hspace{-4.7cm} h_k^{\prime}(\tau) = A_0\left(u_k(\tau, 1)-\frac{\sigma(h_k(\tau))}{m_0}\right), \;\;\;\tau\in(0, \hat{T})
 \end{align}
 for all  $\varphi_k \in \text{span}\{\phi_j\},\; j\in\{0,1,\cdots,N-1\}$
 and
 \begin{align}
\label{initial_u}& \alpha_j(0) = (u_{0, k}, \phi_j),\\
\label{initial_h}& h_k(0) = h_0.
\end{align}
Taking in \eqref{a5} and \eqref{a6} as test function $\varphi_k = \phi_j$ for $ j\in\{0, 1,\cdots, N-1\}$, we obtain the following system of ordinary differential equations for the unknown $\alpha = (\alpha_j)_{j = 0, 1,\cdots, N-1}$ and $h_k$:
\begin{align}
&  \sum_{i=0}^{N-1}M_i\alpha_i^{\prime}(\tau) - \frac{h_k^{\prime}}{h_k} \sum_{i=0}^{N-1}K_i\alpha_i + \frac{1}{h_k^2} \sum_{i=0}^{N-1}A_i\alpha_i = \frac{1}{h_k} {\rm Bi} \left( \frac{b(\tau)}{m_0} \phi(0) - \text{H} \alpha \right) -\frac{h_k^{\prime}}{h_k} \alpha =: G_1(\alpha, h_k),\\
\label{a8}& h_k^{\prime}(\tau) = A_0\left(\sum_{i=0}^{N-1}\alpha_i \phi_i(1)-\frac{\sigma(h_k(\tau))}{m_0}\right)=: G_2(\alpha, h_k),
\end{align}
where 
\begin{align}
\label{matrixM}&(M_i)_j := \int_{0}^1  \phi_i \phi_j dy,\\
\label{matrixK}&(K_i)_j := \int_{0}^1 y \frac{\partial \phi_i}{\partial y} \phi_jdy,\\
 \label{matrixA}&(A_i)_j := \int_{0}^1  \frac{\partial \phi_i}{\partial y} \frac{\partial \phi_j}{\partial y} dy.
\end{align}
Firstly, we prove that $G_2$ is Lipschitz. 
Let $(\alpha, h_k)$ and  $(\beta, \tilde{h}_k)$ be two pairs. 
\begin{align}
\label{a7}\left| G_2(\alpha, h_k) - G_2(\beta, \tilde{h}_k) \right| \leq A_0 \left(\sum_{i=0}^{N-1} \left| \alpha_i(\tau)- \beta_i(\tau)\right|\left|\phi_i(1)\right| + \frac{1}{m_0} \left| \sigma(h_k(\tau)) - \sigma(\tilde{h}_k(\tau))\right|\right).
\end{align}
Using (A4) in \eqref{a7}, we get
\begin{align*}
\left| G_2(\alpha, h_k) - G_2(\beta, \tilde{h}_k) \right| &\leq A_0 \left(\sum_{i=0}^{N-1} \left| \alpha_i(\tau)- \beta_i(\tau)\right| \left|\phi_i(1)\right| + \frac{\mathcal{L}}{m_0} \left| h_k(\tau) - \tilde{h}_k(\tau)\right|\right)\\
&\leq \mathcal{M} \left(\sum_{i=0}^{N-1} \left| \alpha_i(\tau)- \beta_i(\tau)\right| + \left| h_k(\tau) - \tilde{h}_k(\tau)\right|\right)\\
&=\mathcal{M}|(\alpha, h_k)- (\beta, \tilde{h}_k)|,
\end{align*}
where $\mathcal{L}$ is a Lipschitz constant and
$$ \mathcal{M}:= \max \left\{ A_0 \max_{0\leq i\leq N-1}|\phi_i(1)|, \frac{A_0 \mathcal{L}}{m_0}\right\}.$$
Thus, $G_2$ is Lipschitz. 
%{\color{red}Also, the function $h_k^{\prime}$ is Lipschitz continuous with respect to $h_k$. Thus, there exists a continuous function $h_k$ which satisfies \eqref{a8} (i.e. existence for the $h_k$). It implies that $h_k$  is bounded.  Is $h_k^{\prime}$ also bounded in $(0, \tilde{T})$ ?}\\
Now, we show that $G_1$ is Lipschitz. 
\begin{align}
  \label{aaa1}  G_1(\alpha, h_k) - G_1(\beta, \tilde{h}_k) = {\rm Bi} \frac{b(\tau)}{m_0} \left( \frac{1}{h_k} - \frac{1}{\tilde{h}_k}\right)\phi(0) - {\rm Bi}\,\text{H} \left( \frac{\alpha}{h_k} - \frac{\beta}{\tilde{h}_k} \right) - \left( \frac{h_k^{\prime}}{h_k} \alpha - \frac{\tilde{h}_k^{\prime}}{\tilde{h}_k}  \beta\right).
\end{align}
Using (A2) in \eqref{aaa1} yields
\begin{align*}
\left| G_1(\alpha, h_k) - G_1(\beta, \tilde{h}_k) \right| \leq & \;\text{Bi} \frac{b^*}{m_0h_k \tilde{h}_k}|h_k - \tilde{h}_k||\phi(0)|+ \text{Bi}\,\text{H} \left| \frac{\alpha}{h_k} - \frac{\beta}{\tilde{h}_k} \right|   \\
& + \left|\frac{h_k^{\prime}}{h_k} \alpha - \frac{\tilde{h}_k^{\prime}}{\tilde{h}_k}  \beta \right|\\
 =& \sum_{\ell=1}^3 I_{\ell},
\end{align*}
where 
\begin{align*}
I_1 &:= \text{Bi} \frac{b^*}{m_0h_k \tilde{h}_k}|h_k - \tilde{h}_k||\phi(0)| \leq \text{Bi} \frac{b^*}{m_0h_k \tilde{h}_k}|h_k - \tilde{h}_k|,\\
I_2&:= \text{Bi}\,\text{H} \left| \frac{\alpha}{h_k} - \frac{\beta}{\tilde{h}_k} \right| \leq \text{Bi}\,\text{H} \left(|\alpha| \frac{| h_k - \tilde{h}_k|}{h_k \tilde{h}_k} + \frac{|\alpha - \beta|}{ \tilde{h}_k}\right),\\
\nonumber I_3&:= \left|\frac{h_k^{\prime}}{h_k} \alpha - \frac{\tilde{h}_k^{\prime}}{\tilde{h}_k}  \beta \right|\\
\nonumber&= \left|h_k^{\prime}  \left(\frac{\alpha}{h_k} - \frac{\beta}{\tilde{h}_k}  \right) + \frac{\beta}{\tilde{h}_k} \left( h_k^{\prime} - \tilde{h}_k^{\prime} \right)\right| \\
\nonumber& \leq \left| h_k^{\prime}\right| \left| \frac{\alpha}{h_k} - \frac{\beta}{\tilde{h}_k} \right| + \mathcal{L}\frac{|\beta|}{|\tilde{h}_k|} \left| h_k - \tilde{h}_k  \right|\\
& \leq \left| h_k^{\prime}\right| \left(|\alpha| \frac{| h_k - \tilde{h}_k|}{h_k \tilde{h}_k} + \frac{|\alpha - \beta|}{ \tilde{h}_k}\right) + \mathcal{L}\frac{|\beta|}{|\tilde{h}_k|} \left| h_k - \tilde{h}_k  \right|.
\end{align*}
This shows that $G_1$ is Lipschitz continuous. By {\clb standard arguments for systems of ordinary differential equations}, the problem \eqref{initial_u}-\eqref{a8} has a unique solution $$(\alpha, h_k )\in C^1([0, \hat{T}])^N \times W^{1, \infty}(0, \hat{T}).$$
We now prove the uniform estimate for the solution $u$ to the finite dimensional problem.\\
%{\color{red} Can we bound $\alpha^N$, $\beta^N$, $h_k^{\prime}$,  and  $h_k$ ?}\\
Taking $\varphi = u_k$ in \eqref{a5} yields
\begin{align}
\nonumber\frac{1}{2}\frac{d}{d\tau} \| u_k(\tau) \|^2 &+ \frac{1}{(h_k(\tau))^2} \left\Vert\frac{\partial u_k(\tau)}{\partial y} \right\Vert^2 = \int_0^1 \frac{h_k^{\prime}}{h_k} y \frac{\partial u_k}{\partial y} u_k dy\\
\label{uni}& + \frac{1}{h_k(\tau)} \text{Bi}\left(\frac{b(\tau)}{m_{0}}- \text{H}u_k(\tau, 0)\right) u_k(\tau, 0) +\frac{h_k^{\prime}(\tau)}{h_k(\tau)}u_k(\tau, 1)u_k(\tau, 1). 
\end{align}
Using H\"{o}lder’s inequality for the first term on the right hand side of \eqref{uni}, it holds that
\begin{align}
\nonumber\frac{1}{2}\frac{d}{d\tau} \| u_k(\tau) \|^2 + \frac{1}{(h_k(\tau))^2} \left\Vert\frac{\partial u_k(\tau)}{\partial y} \right\Vert^2 & \leq \frac{|h_k^{\prime}|}{h_k} \left\Vert \frac{\partial u_k}{\partial y }\right\Vert_{L^2(\Omega)}\left\Vert u_k\right\Vert_{L^2(\Omega)}\\
&\label{aa3} + \frac{\rm{Bi}}{h_k} \frac{b^*}{m_0}|u_k(\tau, 0)| +
 \frac{\left|h_k^{\prime}\right|}{h_k} \left| u_k(\tau, 1) \right|^2.
\end{align}
We note here that, by the Sobolev's embedding inequality in one space dimension, it holds
\begin{align}
\label{aa2}| \vartheta (\tau, y)|^2 \leq C_e \left\Vert \vartheta (\tau) \right\Vert_{H^1(0, 1)}  \left\Vert \vartheta (\tau) \right\Vert_{L^2(0, 1)}\;\; \text{for} \;\vartheta \in H^1(0, 1) \;\text{and}\; y \in [0, 1],    
\end{align}
where $C_e$ is a positive constant. Using \eqref{aa2},  the third term on the right hand side of \eqref{aa3} becomes
\begin{align}
   \nonumber  \frac{\left|h_k^{\prime}\right|}{h_k} \left| u_k(\tau, 1) \right|^2 &\leq C_e\frac{\left\Vert h_k^{\prime}\right\Vert_{L^{\infty}(S_{\hat{T}})}}{h_0} \left\Vert u_k(\tau) \right\Vert_{H^1(0, 1)} \left\Vert u_k(\tau) \right\Vert_{L^2(0, 1)}\\
   \label{aa4}  &\leq C_e\frac{\left\Vert h_k^{\prime}\right\Vert_{L^{\infty}(S_{\hat{T}})}}{h_0} \left( \left\Vert \frac{\partial u_k(\tau)}{\partial y} \right\Vert_{L^2(0, 1)} \left\Vert u_k(\tau) \right\Vert_{L^2(0, 1)} +   \left\Vert u_k(\tau) \right\Vert_{L^2(0, 1)}^2\right).
\end{align}
 Using \eqref{aa4}, \eqref{aa3}  becomes
\begin{align}
    \nonumber\frac{1}{2}\frac{d}{d\tau} \| u_k \|^2 + \frac{1}{(h_k)^2} \left\Vert\frac{\partial u_k}{\partial y} \right\Vert^2 \leq & (1+C_e)\frac{\left\Vert h_k^{\prime}\right\Vert_{L^{\infty}(S_{\hat{T}})}}{h_0} \left\Vert \frac{\partial u_k}{\partial y} \right\Vert_{L^2(0, 1)} \left\Vert u_k \right\Vert_{L^2(0, 1)}\\
    \label{aa5}&+ C_e\frac{\left\Vert h_k^{\prime}\right\Vert_{L^{\infty}(S_{\hat{T}})}}{h_0}   \left\Vert u_k \right\Vert_{L^2(0, 1)}^2 + \frac{1}{h_0} \frac{b^*}{m_0}\|u_k\|_{H^1(0, 1)}.
\end{align}
Using Young's inequality, \eqref{aa5} leads to 
\begin{align}
 \nonumber  \frac{1}{2}\frac{d}{d\tau} \| u_k \|^2 + \frac{1}{2L^2} \left\Vert\frac{\partial u_k}{\partial y} \right\Vert^2  \leq & (1+C_e)\frac{\left\Vert h_k^{\prime}\right\Vert_{L^{\infty}(S_{\hat{T}})}}{h_0} \left( \xi \left\Vert \frac{\partial u_k}{\partial y} \right\Vert_{L^2(0, 1)}^2 + c_{\xi}\left\Vert u_k \right\Vert_{L^2(0, 1)}^2 \right) \\
  \nonumber  & +  C_e\frac{\left\Vert h_k^{\prime}\right\Vert_{L^{\infty}(S_{\hat{T}})}}{h_0}   \left\Vert u_k \right\Vert_{L^2(0, 1)}^2 + \xi \left\Vert \frac{\partial u_k}{\partial y} \right\Vert_{L^2(0, 1)}^2 \\
  \nonumber  &+ \xi \left\Vert u_k \right\Vert_{L^2(0, 1)}^2+ \frac{c_{\xi}}{h_0^2} \frac{(b^*)^2}{m_0^2}.   
\end{align}
Finally, we get the following inequality
\begin{align}
    \label{aa7}\frac{1}{2}\frac{d}{d\tau} \| u_k \|^2 +  M_1   \left\Vert\frac{\partial u_k}{\partial y} \right\Vert^2 \leq M_2 \left\Vert  u_k \right\Vert_{L^2(0, 1)}^2 + M_3,
\end{align}
where 
\begin{align*}
&M_1: = \frac{1}{2L^2} - \left( (1+C_e)\frac{\left\Vert h_k^{\prime}\right\Vert_{L^{\infty}(S_{\hat{T}})}}{h_0} +1\right) \xi, \\
&M_2:= \frac{\left\Vert h_k^{\prime}\right\Vert_{L^{\infty}(S_{\hat{T}})}}{h_0} \left(c_{\xi} +  C_e (c_{\xi} + 1) \right) + \xi,\\
& M_3:= \frac{c_{\xi}}{h_0^2} \frac{(b^*)^2}{m_0^2}.   
\end{align*}
Choosing a sufficiently small $\xi$ with $M_1>0$ and then applying Gronwall's inequality  gives the following inequality holds
\begin{align}
 \label{aa9}  \| u_k(\tau) \|^2  \leq c(\hat{T},h_0, C_e) \left(  \| u_k(0) \|^2 + M_3 \hat{T}   \right),   
\end{align}
for all $0\leq \tau \leq \hat{T}.$
Since $\| u_k(0) \|^2 \leq \| u_{0,k} \|^2$, \eqref{aa9} yields
\begin{align}
    \label{aa8} \max_{0\leq \tau\leq \hat{T}} \| u_k(\tau) \|^2  \leq \tilde{c}. 
\end{align}
Integrating \eqref{aa7} from $0$ to $\hat{T}$ and employ the inequality \eqref{aa8} to get
\begin{align*}
\int_0^{\hat{T}}\left\Vert\frac{\partial u_k}{\partial y} \right\Vert^2 d\tau \leq \tilde{c}.
\end{align*}
This concludes the proof of \eqref{energy}. \end{proof}
\begin{remark}
 The entries of the  matrices $M, \;K$ and $A$ given in \eqref{matrixM}, \eqref{matrixK} and \eqref{matrixA} are computed explicitly benefiting of the structure of the basis elements $\phi_j \in V_k$, usually piecewise polynomials of some preset degree defined in $\Omega$; see \cite{nepal2021moving} for the explicit form of the matrix $K$ and $A$ when using as basis  piecewise linear functions.
\end{remark}
\section{Main results} \label{mainresults}
In this Section, we prove \textit{a priori} and \textit{a posteriori} error estimates between the weak solution to $(P)$ and weak solution to a
semi-discrete version of $(P)$. The discretization in space is done via the finite element method \cite{larsson2008partial}.
%\subsection{Uniform estimate for the discretized problems}
%In this subsection, we prove uniform estimates for the solutions to the finite-dimensional problems. We use this result in priori error estimates.
\begin{theorem} \label{apriori}
(A priori error estimate) Assume {\rm (A\ref{A1})--(A\ref{A5})} hold. Additionally, take $u_0 \in H^2(0, 1)$. Let $(u, h)$ and $(u_k, h_k)$ be the corresponding weak solutions to problem $(P)$ and $(P_d)$ in the sense of Definition \ref{D1} and Definition \ref{D2}, respectively. Then there exists a constant  $c>0$ (not depending on k) such that
\begin{align}
\label{aprioriestimate}\| u - u_k \|^2_{L^{\infty}(S_{\hat{T}}, L^2(0, 1)))\cap L^2(S_{\hat{T}}, H^1(0, 1))} + \| h- h_k\|^2_{H^1(S_{\hat{T}})} \leq ck^2.  
\end{align}
\end{theorem}
\begin{proof}
We consider the time interval $S_{\hat{T}}$ on which both continuous and discrete solutions to \eqref{aa17}--\eqref{aa22} exist and are uniquely defined.  Let $e:= u - u_k$ and $h-h_k$ be the pointwise errors of the approximation.  By subtracting \eqref{a19} from \eqref{a17} and choosing $\varphi = v_k \in V_k$, we obtain the following identity:
\begin{align}
\nonumber \left( \frac{\partial u}{\partial \tau}, v_k \right) - & \left( \frac{\partial u_k}{\partial \tau}, v_k \right)  +  \frac{1}{h^2} \left( \frac{\partial{u}}{\partial{y}}, \frac{\partial v_k }{\partial y}  \right) - 
  \frac{1}{h_k^2} \left(\frac{\partial u_k}{\partial y}, \frac{\partial v_k }{\partial y} \right)\\
  \nonumber & - \left( \frac{h^{\prime}}{h} \int_0^1 y \frac{\partial u}{\partial y} v_k dy - \frac{h_k^{\prime}}{h_k} \int_0^1 y \frac{\partial u_k}{\partial y} v_k dy  \right) + \frac{h^{\prime}}{h} u(\tau, 1) v_k(1) -  \frac{h_k^{\prime}}{h_k} u_k(\tau, 1) v_k(1) \\
 \label{a23}&- \left( \frac{1}{h} \text{Bi} \left( \frac{b(\tau)}{ m_0} - \text{H} u(\tau, 0)\right) v_k(0) -  \frac{1}{h_k} \text{Bi} \left( \frac{b(\tau)}{ m_0} - \text{H} u_k(\tau, 0)\right) v_k(0) \right) = 0,
  \end{align}
  which holds for all $v_k \in V_k$ and for almost every $\tau \in S_{\hat{T}}$.\\
 Arranging conveniently the terms in \eqref{a23} yields
  \begin{align}
  \nonumber \left( \frac{\partial e}{\partial \tau}, v_k \right)  &+  \frac{1}{h^2} \left( \frac{\partial{e}}{\partial{y}}, \frac{\partial v_k }{\partial y}  \right) - \left(
  \frac{1}{h_k^2} -  \frac{1}{h^2} \right)  \left(\frac{\partial u_k}{\partial y}, \frac{\partial v_k }{\partial y} \right)\\
  \nonumber & - \left( \frac{h^{\prime}}{h} \int_0^1 y \frac{\partial e}{\partial y} v_k dy +  \left( \frac{h^{\prime}}{h}- \frac{h_k^{\prime}}{h_k} \right) \int_0^1 y \frac{\partial u_k}{\partial y} v_k dy  \right) \\
   \nonumber&+ \frac{h^{\prime}}{h} e(\tau, 1) v_k(1) + \left( \frac{h^{\prime}}{h} -  \frac{h_k^{\prime}}{h_k}\right) u_k(\tau, 1) v_k(1) \\
  \label{a24}&- \left( \text{Bi} \frac{b(\tau)}{m_0}\left( \frac{1}{h} - \frac{1}{h_k}\right) v_k(0) - \text{Bi} \;\text{H}    \left(  \frac{u(\tau, 0)}{h}  -  \frac{u_k(\tau, 0)}{h_k} \right) v_k(0) \right) = 0.
  \end{align}
  In \eqref{a24}, we take as test function $v_k: = w_k - u_k \in V_k$  and use the decomposition $v_k = (w_k - u) + e$. Then \eqref{a24} becomes 
  \begin{align}
  \nonumber \left( \frac{\partial e}{\partial \tau}, e \right)  & + \left( \frac{\partial e}{\partial \tau}, w_k - u \right) +  \frac{1}{h^2} \left( \frac{\partial{e}}{\partial{y}}, \frac{\partial e }{\partial y}  \right) +   \frac{1}{h^2} \left( \frac{\partial{e}}{\partial{y}}, \frac{\partial }{\partial y} (w_k - u) \right)\\
   \nonumber  &- \left(
  \frac{1}{h_k^2} -  \frac{1}{h^2} \right)  \left(\frac{\partial u_k}{\partial y}, \frac{\partial  }{\partial y}( w_k - u_k) \right)
  - \left( \frac{h^{\prime}}{h} \int_0^1 y \frac{\partial e}{\partial y} ( w_k - u_k) dy +  \left( \frac{h^{\prime}}{h}- \frac{h_k^{\prime}}{h_k} \right) \int_0^1 y \frac{\partial u_k}{\partial y}( w_k - u_k) dy  \right) \\
  \nonumber&+ \frac{h^{\prime}}{h} e(\tau, 1)( w_k(1) - u_k(1)) + \left( \frac{h^{\prime}}{h} -  \frac{h_k^{\prime}}{h_k}\right) u_k(\tau, 1) (w_k(1) - u_k(1)) \\
 \nonumber &- \left( \text{Bi} \frac{b(\tau)}{m_0}\left( \frac{1}{h} - \frac{1}{h_k}\right) (w_k(0) - u_k(0)) - \text{Bi} \;\text{H}    \left(  \frac{u(\tau, 0)}{h}  -  \frac{u_k(\tau, 0)}{h_k} \right) (w_k(0) - u_k(0)) \right) = 0.
\end{align}
Therefore, we can write
    \begin{align}
  \nonumber  \frac{1}{2} \frac{d}{d \tau}  \lVert e \rVert^2 + \frac{1}{h^2} \left\Vert \frac{\partial e}{\partial y} \right\Vert^2 &\leq \left\Vert \frac{\partial e}{\partial \tau} \right\Vert \left\Vert u - w_k \right\Vert  +  \frac{1}{h^2} \left\Vert \frac{\partial e}{\partial y} \right\Vert \left\Vert\frac{\partial}{\partial y} (u - w_k) \right\Vert \\
  \nonumber &+ |h - h_k| \frac{h + h_k}{ h^2 h_k^2} \left\Vert \frac{\partial u_k}{\partial y} \right\Vert \left\Vert  \frac{\partial}{\partial y}(w_k - u_k)  \right\Vert + \frac{h^{\prime}}{ h} \left\Vert \frac{\partial e}{\partial y}  \right\Vert  \left\Vert w_k - u_k \right\Vert \\
 \nonumber  & + \left|  \frac{h^{\prime}}{h} -   \frac{h^{\prime}_k}{h_k} \right|  \left\Vert \frac{\partial u_k}{\partial y} \right\Vert \left\Vert  \frac{\partial}{\partial y}(w_k - u_k)  \right\Vert + \frac{h^{\prime}}{h} |e(\tau, 1)| |(w_k(1) - u(1) ) + e(\tau, 1)|\\
 \nonumber  & +  \left|  \frac{h^{\prime}}{h} -   \frac{h^{\prime}_k}{h_k} \right| |u_k(\tau, 1)| |(w_k(1) - u(1)) + e(\tau, 1)| \\
 \label{a26} & + \text{Bi} \frac{b^*}{m_0} \left|  \frac{1}{h} -   \frac{1}{h_k} \right| |w_k(0) - u_k(0)| + \text{Bi}\; \text{H} \left|  \frac{u(\tau, 0)}{h}  -  \frac{u_k(\tau, 0)}{h_k}   \right| |w_k(0) - u_k(0)|.
  \end{align}
  To bound some terms on the right hand side in \eqref{a26}, we introduce the strictly positive constant $c_{\ell} < \infty, \ell \in \{1, 2, \cdots, 5\}$.  The value for these constants is not explicitly written, but can be calculated.  Before proceeding further, we collect two useful estimates in Remark \ref{Remark1}. 
  \begin{remark}
  \label{Remark1}
  There exist constants $c_2,\; c_5>0$ such that 
  \begin{align*}
  (1)\;\; & \left|  \frac{h^{\prime}}{h} -   \frac{h^{\prime}_k}{h_k} \right| \leq c_2 (|h - h_k| + |h^{\prime} - h^{\prime}_k|)\\
 (2)\;\; &  \left(  \frac{u(0)}{h}  -  \frac{u_k(0)}{h_k} \right) = \frac{1}{h} (e(0)) + \frac{u_k(0)}{h_k}(h_k - h) \leq c_5(|e(0)| + |h - h_k|).
  \end{align*}
  \end{remark}
 Making use of Remark \ref{Remark1}, \eqref{a26} becomes
   \begin{align}
  \nonumber  \frac{1}{2} \frac{d}{d \tau}  \lVert e \rVert^2 + \frac{1}{h^2} \left\Vert \frac{\partial e}{\partial y} \right\Vert^2 &\leq \left\Vert \frac{\partial e}{\partial \tau} \right\Vert \left\Vert u - w_k \right\Vert  +  \frac{1}{h^2} \left\Vert \frac{\partial e}{\partial y} \right\Vert \left\Vert\frac{\partial}{\partial y} (u - w_k) \right\Vert \\
  \nonumber &+ c_1|h - h_k| \left\Vert \frac{\partial u_k}{\partial y} \right\Vert \left( \left\Vert  \frac{\partial}{\partial y}(w_k - u)  \right\Vert + \left\Vert \frac{\partial e}{\partial y}\right\Vert \right) +\frac{h^{\prime}}{h} \left\Vert \frac{\partial e}{\partial y} \right\Vert  \left( \left\Vert  w_k - u  \right\Vert + \left\Vert e\right\Vert \right) \\
  \nonumber  & +  c_2(|h - h_k| + |h^{\prime} - h^{\prime}_k|) \left\Vert \frac{\partial u_k}{\partial y} \right\Vert  \left( \left\Vert  \frac{\partial}{\partial y}(w_k - u)  \right\Vert + \left\Vert \frac{\partial e}{\partial y}\right\Vert \right)\\
  \nonumber &+ {\clb \frac{h^{\prime}}{h}}|e(1)| (|w_k(1) - u(1)| +|e(1)|) \\
\nonumber &  +  c_3 (|h - h_k| + |h^{\prime} - h^{\prime}_k|) |u_k(\tau, 1)| (|w_k(1) - u(1)| + |e(1)|) \\
 \nonumber & + c_4\text{Bi} \frac{b^*}{m_0} \left| h - h_k\right| (|w_k(0) - u(0)| + |e(0)|)\\
\nonumber  & + c_5 \text{Bi}\; \text{H} (| e(0)| + |h - h_k|) (|w_k(0) - u(0)|+ |e(0)|) = \sum_{\ell  = 1}^{9} I_{\ell}.
  \end{align}
  We set $w_k := I_ku$, where $I_ku$ is the Lagrange interpolation of $u$. By using Lemma \ref{lemma1}, Young's inequality \eqref{a3} and interpolation inequality \eqref{a33}, we obtain the following estimates: 
  \begin{align*}
  I_1 & :=  \left\Vert \frac{\partial e}{\partial \tau} \right\Vert \left\Vert u - w_k \right\Vert \leq  \left\Vert \frac{\partial e}{\partial \tau} \right\Vert \gamma_1 k^2 \| u\|_{H^2(0,1)} \leq  \frac{1}{2}\left\Vert \frac{\partial e}{\partial \tau} \right\Vert^2 k^2 + \frac{\gamma_1^2 k^2}{2} \| u\|^2_{H^2(0,1)},\\
  I_2 & := \frac{1}{h^2} \left\Vert \frac{\partial e}{\partial y} \right\Vert \left\Vert\frac{\partial}{\partial y} (u - w_k) \right\Vert  \leq \frac{1}{h^2} \left\Vert \frac{\partial e}{\partial y} \right\Vert \gamma_2 k\left\Vert u\right\Vert_{H^2(0,1)} \leq \frac{\xi}{h^2} \left\Vert \frac{\partial e}{\partial y} \right\Vert^2 + c_{\xi} \gamma_2^2 k^2 \frac{1}{h^2} \| u\| ^2_{H^2(0,1)}, \\
 \nonumber I_3 &:= c_1|h - h_k| \left\Vert \frac{\partial u_k}{\partial y} \right\Vert \left( \left\Vert  \frac{\partial}{\partial y}(w_k - u)  \right\Vert + \left\Vert \frac{\partial e}{\partial y}\right\Vert \right)\\
 \nonumber & \leq  c_1|h - h_k| \left\Vert \frac{\partial u_k}{\partial y} \right\Vert \left( \gamma_2 k \left\Vert  u  \right\Vert_{H^2(0, 1)} + \left\Vert \frac{\partial e}{\partial y}\right\Vert \right)\\
  &  \leq \rho |h - h_k|^2 {\clb \left\Vert \frac{\partial u_k}{\partial y}\right\Vert^2} + c_{\rho} c_1^2 \gamma_2^2 k^2  \left\Vert  u  \right\Vert^2_{H^2(0, 1)} + c_{\hat{\rho}} c_1^2 |h - h_k|^2  { \clb \left\Vert \frac{\partial u_k}{\partial y}\right\Vert^2} h^2 + \hat{\rho}  \frac{1}{h^2} \left\Vert \frac{\partial e}{\partial y}\right\Vert^2. 
  \end{align*}
  {\clb We observe that if $u_0 \in H^1(0, 1)$, then it also holds that $u_k \in H^1(S_{\hat{T}}, V_k)$. Hence, we can control terms like $\left\Vert \frac{\partial u_k}{\partial y}\right\Vert^2$ via
  \begin{align}
  \label{lifting}    \max_{0\leq \tau \leq \hat{T}} \| u_k\|^2_{H^1(0, 1)} + \int_0^{\hat{T}} \| u_k\|^2_{H^2(0, 1)} d\tau \leq \hat{c}_1.
  \end{align}
  }
  
  Therefore, we get 
  \begin{align*}
I_3 & \leq  \rho \hat{c}_1 |h - h_k|^2  + c_{\rho} c_1^2 \gamma_2^2 k^2  \left\Vert  u  \right\Vert^2_{H^2(0, 1)} +  c_{\hat{\rho}}  \hat{c}_1 c_1^2 |h - h_k|^2  h^2 + \hat{\rho}  \frac{1}{h^2} \left\Vert \frac{\partial e}{\partial y}\right\Vert^2,  \\
 \nonumber I_4 &:= \frac{h^{\prime}}{h} \left\Vert \frac{\partial e}{\partial y} \right\Vert  \left( \left\Vert  w_k - u  \right\Vert + \left\Vert e\right\Vert \right)\\
 \nonumber& \leq \frac{h^{\prime}}{h} \left\Vert \frac{\partial e}{\partial y} \right\Vert  \left( \gamma_1 k^2 \left\Vert u  \right\Vert_{H^2(0,1)} + \left\Vert e\right\Vert \right) \\
\nonumber &\leq\zeta \frac{1}{h^2} \left\Vert \frac{\partial e}{\partial y} \right\Vert ^2 + c_{\zeta} (h^{\prime})^2 \left( \gamma_1 k^2 \left\Vert u  \right\Vert_{H^2(0,1)} + \left\Vert e\right\Vert \right)^2  \\
 & \leq\zeta \frac{1}{h^2} \left\Vert \frac{\partial e}{\partial y} \right\Vert ^2 +  2c_{\zeta} (h^{\prime})^2 \left( \gamma_1^2 k^4 \left\Vert u  \right\Vert^2_{H^2(0,1)} + \left\Vert e\right\Vert^2 \right), \\
 % \frac{1}{2}|e(1)|^2 +\left\Vert\frac{\partial e}{\partial y}\right\Vert\| w_k - u \|\\
 \nonumber I_5& :=  c_2(|h - h_k| + |h^{\prime} - h^{\prime}_k|) \left\Vert \frac{\partial u_k}{\partial y} \right\Vert  \left( \left\Vert  \frac{\partial}{\partial y}(w_k - u)  \right\Vert + \left\Vert \frac{\partial e}{\partial y}\right\Vert \right)\\
\nonumber &\leq c_2(|h - h_k| + |h^{\prime} - h^{\prime}_k|) { \clb \left\Vert \frac{\partial u_k}{\partial y}\right\Vert}\left( \gamma_2 k \left\Vert u  \right\Vert_{H^2(0,1)} + \left\Vert \frac{\partial e}{\partial y}\right\Vert \right)\\
 & \leq \xi \left(|h - h_k|^2 { \clb \left\Vert \frac{\partial u_k}{\partial y}\right\Vert^2} + |h^{\prime} - h^{\prime}_k|^2 { \clb \left\Vert \frac{\partial u_k}{\partial y}\right\Vert^2}\right) + c_{\xi} c_2^2 \gamma_2^2 k^2 \left\Vert u  \right\Vert^2_{H^2(0,1)}\\
 &+ \hat{\xi} \frac{1}{h^2}\left\Vert \frac{\partial e}{\partial y}\right\Vert^2 + c_{\hat{\xi}} c_2^2  h^2\left(|h - h_k|^2 { \clb \left\Vert \frac{\partial u_k}{\partial y}\right\Vert^2} + |h^{\prime} - h^{\prime}_k|^2 { \clb \left\Vert \frac{\partial u_k}{\partial y}\right\Vert^2}\right),\\
 & \leq \xi \hat{c}_1 \left(|h - h_k|^2 + |h^{\prime} - h^{\prime}_k|^2 \right) + c_{\xi} c_2^2 \gamma_2^2 k^2 \left\Vert u  \right\Vert^2_{H^2(0,1)}\\
 &+ \hat{\xi} \frac{1}{h^2}\left\Vert \frac{\partial e}{\partial y}\right\Vert^2 + c_{\hat{\xi}} \hat{c}_1 c_2^2  h^2\left(|h - h_k|^2  + |h^{\prime} - h^{\prime}_k|^2 \right),\\
\nonumber I_6 &:= \frac{h^{\prime}}{h} |e(1)| (|w_k(1) - u(1)| + |e( 1)|)\\
\nonumber & =  \frac{h^{\prime}}{h} |e(1)|^2 +   \frac{h^{\prime}}{h} |e(1)| |w_k(1) - u(1)|\\
\nonumber & =  \frac{h^{\prime}}{h} |e(1)|^2  + \frac{h^{\prime}}{h}  \left( \frac{ |e(1)|^2}{2} + \frac{|w_k(1) - u(1)|^2}{2}\right)\\
\nonumber & =  \frac{3}{2}\frac{h^{\prime}}{h} |e(1)|^2  + \frac{h^{\prime}}{h}  \frac{|w_k(1) - u(1)|^2}{2}\\
 \nonumber &\leq {\clb \frac{3}{2} \frac{h^{\prime}}{h} \hat{c} \left\Vert e \right\Vert ^{2 \theta} \left \Vert  e \right \Vert_{H^1(0,1)}^{2(1-\theta)} +  \frac{h^{\prime}}{2h}  \left(\gamma_1 k^2 + \gamma_3 k^{1+\theta} \right)^2 \left\Vert  u\right\Vert_{H^2(0,1)}^{2} }  \\
& \leq {\clb \frac{3}{2} \left( \frac{\xi}{h^2} \left\Vert \frac{\partial e}{\partial y}\right\Vert^2 + (\frac{\xi}{h^2}+ c_{\xi} \hat{c}^2 (h^{\prime})^2) \left\Vert e\right\Vert^2   \right) + \frac{h^{\prime}}{2h}  \hat{c}^2 \left(  \gamma_1 k^2 + \gamma_3 k^{1+\theta} \right)^2 \left\Vert  u\right\Vert_{H^2(0,1)}^{2},}\\
I_7 \nonumber& :=  c_3 (|h - h_k| + |h^{\prime} - h^{\prime}_k|) |u_k(\tau, 1)| (|w_k(1) - u(1)| + |e(1)|) \\
\nonumber&\leq  c_3 (|h - h_k| + |h^{\prime} - h^{\prime}_k|) \hat{c} \left\Vert u_k \right\Vert ^{1-\theta} \left\Vert \frac{ \partial u_k}{\partial y} \right\Vert ^{\theta} (|w_k(1) - u(1)| + |e(1)|) \\
\nonumber&  \leq {\clb c_3 (|h - h_k| + |h^{\prime} - h^{\prime}_k|) \hat{c}^2 \left\Vert u_k \right\Vert_{H^1(0,1)} \left(  \gamma_1 k^2 + \gamma_3 k^{1+\theta} \right) \|u\|_{H^2(0,1)} }\\
  \nonumber & {\clb +  c_3 (|h - h_k| + |h^{\prime} - h^{\prime}_k|) \hat{c}^2 \left\Vert u_k \right\Vert_{H^1(0,1)}  \| e\| ^{\theta} \left\Vert e \right\Vert_{H^1(0, 1)} ^{1-\theta}}  \\
\nonumber& \leq {\clb c_3 \hat{c}^2 (|h - h_k| + |h^{\prime} - h^{\prime}_k|) \left\Vert u_k \right\Vert_{H^1(0, 1)}  \left(  \gamma_1 k^2 + \gamma_3 k^{1+\theta} \right) \| u\|_{H^2(0,1)}} \\
\nonumber &  {\clb + c_3 \hat{c}^2 (|h - h_k| + |h^{\prime} - h^{\prime}_k|) \left\Vert u_k \right\Vert_{H^1(0,1)} \| e\| ^{\theta} \left\Vert e \right\Vert_{H^1(0, 1)} ^{1-\theta}} \\
\nonumber&  \leq {\clb \xi \left( |h - h_k| + |h^{\prime} - h_k^{\prime}| \right)^2 \left\Vert u_k \right\Vert_{H^1(0,1)}^2 + c_{\xi} \left(c_3 \hat{c}^2(  \gamma_1 k^2 + \gamma_3 k^{1+\theta} )\right)^2 \| u\|_{H^2(0,1)}^2}\\
\nonumber& {\clb + \bar{\xi} \left( |h - h_k| + |h^{\prime} - h_k^{\prime}| \right)^2 \left\Vert u_k\right\Vert_{H^1(0,1)}^2 + c_{\bar{\xi}} c_3^2 \hat{c}^4 \| e\| ^{2\theta} \left\Vert e \right\Vert_{H^1(0, 1)} ^{2(1-\theta)}}\\
\nonumber&\leq {\clb 2\xi \hat{c}_1 \left( |h - h_k| ^2+ |h^{\prime} - h_k^{\prime}|^2 \right) + c_{\xi} \left(c_3 \hat{c}^2(  \gamma_1 k^2 + \gamma_3 k^{1+\theta} )\right)^2 \| u\|_{H^2(0,1)}^2}\\
&{\clb + 2\bar{\xi} \hat{c}_1 \left( |h - h_k|^2 + |h^{\prime} - h_k^{\prime}|^2 \right)  + \hat{\xi} \frac{1}{h^2}  \left\Vert \frac{\partial e}{\partial y} \right\Vert^2 + \left( \frac{\hat{\xi}}{h^2} + c_{\hat{\xi}} c_{\bar{\xi}}^2  c_3^4 \hat{c}^8 h^2 \right)\| e\|^{2},}\\
\nonumber I_8& :=  c_4\text{Bi} \frac{b^*}{m_0} \left| h - h_k\right| (|w_k(0) - u(0)| + |e(0)|)\\
\nonumber& \leq {\clb c_4 \text{Bi} \frac{b^*}{m_0} \left| h - h_k\right| \left( \hat{c} (\gamma_1 k^2 + \gamma_3 k^{1+\theta}) \left\Vert u \right\Vert_{H^2(0, 1)}  + \hat{c}\left\Vert e  \right\Vert^{\theta} \left\Vert e\right\Vert ^{1 -\theta}_{H^1(0, 1)} \right)}\\
\nonumber&\leq {\clb \xi |h - h_k|^2 + c_{\xi} c_4^2 \hat{c}^2 \text{Bi}^2\frac{(b^*)^2}{m_0^2} (\gamma_1 k^2 + \gamma_3 k^{1+\theta})^2  \| u\|_{H^2(0,1)}^2} \\ 
\nonumber&  {\clb +  \hat{\xi} |h - h_k|^2 + c_{\hat{\xi}} c_4^2 \hat{c}^2 \text{Bi}^2\frac{(b^*)^2}{m_0^2} \left\Vert e  \right\Vert^{2\theta} \left\Vert e\right\Vert ^{2(1 -\theta)}_{H^1(0, 1)}} \\
\nonumber& \leq {\clb \xi |h - h_k|^2 + c_{\xi} c_4^2 \hat{c}^2 \text{Bi}^2\frac{(b^*)^2}{m_0^2} (\gamma_1 k^2 + \gamma_3 k^{1+\theta})^2 \| u\|_{H^2(0,1)}^2} \\ 
& {\clb +  \hat{\xi} |h - h_k|^2 + \bar{\xi}  \frac{1}{h^2} \left\Vert \frac{\partial e}{\partial y}  \right\Vert^{2} +  \left( \frac{\bar{\xi}}{h^2}+ c_{\bar{\xi}}c_{\hat{\xi}}^2 c_4^4 \hat{c}^4 \text{Bi}^4\frac{(b^*)^4}{m_0^4} h^2 \right) \left\Vert e\right\Vert^{2}. }
\end{align*}
By a similar calculation used to obtain the upper bounds on $I_6$ and $I_8$, we get
\begin{align*}
\nonumber I_9&:= c_5 \text{Bi}\; \text{H} (| e(0)| + |h - h_k|) (|w_k(0) - u(0)|+ |e(0)|)\\
\nonumber & \leq {\clb \frac{3}{2} \left( \frac{\xi}{h^2} \left\Vert \frac{\partial e}{\partial y}\right\Vert^2 + \left( \frac{\xi}{h^2} + c_{\xi} \hat{c}^2 c_5^2 \text{Bi}^2 \text{H} ^2 h^2\right) \left\Vert e\right\Vert^2 \right ) + c_{\xi} c_5^2 \text{Bi}^2 \text{H}^2 \hat{c}^2 (\gamma_1 k^2 + \gamma_3 k^{1+\theta})^2 \left\Vert  u\right\Vert_{H^2(0,1)} ^{2} }\\
& {\clb + (\xi +\hat{\xi}) |h - h_k|^2 + \bar{\xi}  \frac{1}{h^2} \left\Vert \frac{\partial e}{\partial y}  \right\Vert^{2} +  \left(\frac{\bar{\xi}}{h^2} + c_{\bar{\xi}} c_{\hat{\xi}}^2 c_5^4 \hat{c}^4 \text{Bi}^4 \text{H}^4 h^2 \right) \left\Vert e\right\Vert ^{2}.} 
  \end{align*} 
 Finally, we are led to the following structural inequality:
  \begin{align}
 \label{a29} \frac{1}{2} \frac{d}{d \tau}  \lVert e \rVert^2 + A_1 \left\Vert \frac{\partial e}{\partial y} \right\Vert^2 &\leq A_2k^2 + A_3 \left\Vert e \right\Vert^2 + A_4 |h - h_k|^2 
   + A_5 |h^{\prime} - h_k^{\prime}|^2,
 \end{align}
 where
 {\clb \begin{align*}
 A_1 &:= \frac{1}{L^2} \left( 1 - \frac{5}{2}\xi - \hat{\rho}   - \zeta - 2 \hat{\xi}  - 2 \bar{\xi}\right),\\
\nonumber A_2 & := \| u\|_{H^2(0, 1)} \left( \frac{\gamma_1^2}{2} + \frac{1}{h_0^2} c_{\xi} \gamma_2^2 + c_{\rho} c_1^2 \gamma_2^2 + 2 c_{\zeta} \| h^{\prime}\|_{\infty}^2 \gamma_1^2 + c_{\xi} c_2^2 \gamma_2^2 + \frac{\| h^{\prime}\|_{\infty}}{2h_0} \hat{c}^2 (\gamma_1 + \gamma_3)^2  \right.\\
& \left. + c_{\xi} c_3^2 \hat{c}^4 (\gamma_1+ \gamma_3)^2 + c_{\xi} c_4^2 \hat{c}^2 \text{Bi}^2 \frac{(b^*)^2}{m_0^2} (\gamma_1+\gamma_3)^2 + c_{\xi} c_5^2 \hat{c}^2 \text{Bi}^2\; \text{H}^2 (\gamma_1+\gamma_3)^2 \right), \\
 A_3 &:= 2 c_{\zeta} \| h^{\prime}\|_{\infty}^2 + \frac{1}{h_0^2} \left( 3\xi + \hat{\xi} + 2 \bar{\xi} \right)+ \frac{3}{2} c_{\xi} \hat{c}^2 \| h^{\prime}\|^2_{\infty} + c_{\hat{\xi}}  c_{\bar{\xi}}^2 c_3^4 \hat{c}^8 \| h \|^2_{\infty} + c_{\bar{\xi}} c_{\hat{\xi}}^2 \hat{c}^4   \text{Bi}^4 \left( c_4^4 \frac{(b^*)^4}{m_0^4} + c_5^4 \rm{H}^4 \right) \| h\|^2_{\infty}\\
 &+ c_{\xi} \hat{c}^2 c_5^2 \text{Bi}^2\; \text{H}^2  \| h\|^2_{\infty},\\
 A_4 & :=2(\xi + \hat{\xi}) + \hat{c}_1(\rho + \xi)  + 2 \hat{c}_1 (\bar{\xi}+ \xi)  + c_{\hat{\rho}} \hat{c}_1 c_1^2 \| h\|_{\infty}^2  + c_2^2 \hat{c}_1 c_{\hat{\xi}} \|h \|_{\infty}^2,\\
 A_5 &:=  3 \xi \hat{c}_1  + 2\bar{\xi} \hat{c}_1 +  c_2^2 \hat{c}_1 c_{\hat{\xi}} \|h \|_{\infty}^2.
 \end{align*} 
  From \eqref{a18} and \eqref{a20}, we get for all $\tau \in(0, \hat{T})$ the inequality
 \begin{align*}
 | h^{\prime}(\tau) - h^{\prime}_k(\tau) | &\leq A_0 |e(1)| + \frac{1}{m_0} |\sigma(h(\tau)) - \sigma(h_k(\tau))|\\
 %& \leq A_0 c \xi \left| \int_0^1 \frac{\partial e}{ \partial y}(\tau, \xi) d\xi \right| + \frac{1}{m_0} L_1 |h - h_k|\\
 & \leq A_0 \hat{c} \left( \eta \left\Vert e(\tau) \right\Vert_{H^1(0, 1)} + c_{\eta} \| e(\tau)\| \right)+ \frac{\mathcal{L}}{m_0} |h(\tau)- h_k(\tau)|.
 \end{align*}
 Thus, this leads to 
 \begin{align}
 \label{a30} | h^{\prime} - h^{\prime}_k |^2 \leq 3\left(A_0^2\hat{c}^2 \eta^2 \left\Vert \frac{\partial e}{\partial y} \right\Vert^2 + A_0^2 \hat{c}^2 (\eta^2 + c_\eta^2) \| e\|^2 + \frac{\mathcal{L}^2}{m_0^2} |h- h_k|^2 \right).
 \end{align}
 Using \eqref{a30} in \eqref{a29}, we infer that
 \begin{align}
 \label{gr} \frac{d}{d \tau}  \lVert e \rVert^2 + (A_1- 3A_0^2 \hat{c}^2 \eta^2 A_5) \left\Vert \frac{\partial e}{\partial y} \right\Vert^2 \leq A_2k^2 &+ A_6 \left\Vert e \right\Vert^2 + \left(A_4+ 3A_5 \frac{\mathcal{L}^2}{m_0^2}\right) |h - h_k|^2,
 \end{align}
 where $A_6:= A_3 + 3A_0^2 \hat{c}^2 (\eta^2 + c_{\eta}^2) A_5$.
 We choose $\xi>0,\; \hat{\rho}>0,\; \bar{\xi}>0,\; \zeta>0$,\; $\hat{\xi}>0$, and $\eta>0$ sufficiently small  such that $\zeta_1:=A_1- 3A_0^2 \hat{c}^2 \eta^2 A_5 > 0$. Applying Gronwall's inequality (see e.g. Appendix B in \cite{evans1998partial}) gives the following  upper bounds:
 \begin{align}
\nonumber \| e(\tau)\| ^2 & \leq e^{\displaystyle \int_0^\tau A_6  ds}\left( \| e(0) \|^2 +  \int_0^\tau \left(A_2k^2 + \left(A_4+ 3A_5 \frac{\mathcal{L}^2}{m_0^2}\right)|h(s) - h_k(s)|^2\right) ds\right)\\
\nonumber &  \leq c_6(A_0, A_3, A_5,  \hat{T}) \left( k^4 \| u_0\|^2_{H^2(0,1)} + A_2 k^2 \tau + \left(A_4+ 3A_5 \frac{\mathcal{L}^2}{m_0^2}\right) \int_0^\tau  |h(s) - h_k(s)|^2 ds\right)\\
\nonumber  & \leq c_6(A_0, A_3, A_4, A_5, \mathcal{L}, \hat{T}) \left( k^4+ k^2 \hat{T}  +  \| h - h_k\|^2_{L^2(S_{\hat{T}})}\right).
 \end{align}
 Thus, we obtain \begin{align}
  \nonumber \displaystyle\max_{0\leq \tau \leq \hat{T}} \| e(\tau)\|^2 \leq c_6\left ( k^2  + \| h - h_k\|^2_{L^2(S_{\hat{T}})}\right).
 \end{align}
 By using  Young's inequality together with \eqref{a30}, we get the following relations:
 \begin{align}
 \nonumber\frac{d}{d\tau} (|h - h_k|^2) &= 2 (h - h_k) (h^{\prime} - h_k^{\prime})\\
 \nonumber& \leq |h-h_k|^2 + | h^{\prime} - h_k^{\prime}|^2\\
\label{a} & \leq C |h-h_k|^2 + 3A_0^2\hat{c}^2 \eta^2 \left\Vert \frac{\partial e}{\partial y} \right\Vert^2 + 3A_0^2 \hat{c}^2 (\eta^2 + c_\eta^2) \| e\|^2,
 \end{align}
 where $C := 1+ 3 \mathcal{L}^2/m_0^2.$\\
% Applying Gronwall inequality, we finally obtain
 %\begin{align*}
 %| h - h_k|^2 \leq c(\mathcal{L}, \hat{T}) \left( |h(0)- h_k(0)| +  \int_0^\tau 3A_0^2\hat{c}^2 \eta^2 \left\Vert \frac{\partial e}{\partial y} \right\Vert^2 ds + \int_0^\tau 3A_0^2 \hat{c}^2 c_\eta^2 \| e\|^2ds \right)\\
 %\leq  c(\mathcal{L}, \hat{T}) \left( |h(0)- h_k(0)| +  \zeta_2 \| e\|^2_{L^2(S_{\hat{T}}, H^1(0, 1))} \right),
 %\end{align*}
 %where $\zeta_2:= 3A_0^2\hat{c}^2 \max\{\eta^2, c_\eta^2\}.$
  Let $\delta>0$ be any positive real number. Adding $\delta \frac{d}{d\tau} |h - h_k|^2$ on both sides of \eqref{gr} and using \eqref{a} yields
 \begin{align}
 \nonumber \frac{d}{d \tau} \left( \lVert e \rVert^2  + \delta | h-h_k|^2\right)+ (\zeta_1 - 3\delta \hat{c}^2 A_0^2\eta^2) \left\Vert \frac{\partial e}{\partial y} \right\Vert^2 \leq A_2 k^2 & + (A_6 + 3 \delta A_0^2 \hat{c}^2 (\eta^2 + c_{\eta}^2)) \|e\|^2 \\
 &+ \left(A_4 + 3 A_5 \frac{\mathcal{L}^2}{m_0^2} + \delta C\right) | h-h_k|^2. 
 \end{align}
 We choose  $\eta>0$ in such a way that $(\zeta_1 - 3\delta \hat{c}^2 A_0^2\eta^2)>0$. Then it exists a constant $A_7>0$ such that\\
  \begin{align}
 \label{Gronwall}  \frac{d}{d \tau} \left( \lVert e \rVert^2  + \delta | h-h_k|^2\right) \leq A_2k^2 +  A_7( \|e\|^2 + \delta | h-h_k|^2).
 \end{align}}
 Gronwall's inequality applied to \eqref{Gronwall} for the quantity $ \|e\|^2 + \delta | h-h_k|^2$ gives the estimate
 \begin{align}
     \label{p1}
     \|e\|^2 + \delta | h-h_k|^2 \leq c k^2.
 \end{align}
  Integrating \eqref{gr} from $0$ to $\hat{T}$ and using \eqref{p1}  yields
 \begin{align}
 \label{grad}\int_0^{\hat{T}} \left\Vert \frac{\partial e}{\partial y} \right\Vert^2  d\tau  \leq c_7 k^2.
 \end{align}
 Integrating \eqref{a30} from $0$ to $\hat{T}$ and using \eqref{p1} and \eqref{grad} gives the estimate
 \begin{align*}
 \|h^{\prime} - h_k^{\prime}\|^2 \leq ck^2,
 \end{align*}
 which completes the proof of Theorem \ref{apriori}. 
 \end{proof}
\begin{theorem}
\label{post}
	(A posteriori error estimate) Assume {\rm (A\ref{A1})--(A\ref{A5})} hold. Additionally, {\clb take} $u_0 \in H^2(0, 1)$. Let $(u, h)$ and $(u_k, h_k)$ be the corresponding weak solutions to the problem $(P)$ and $(P_d)$ in the sense of Definition \ref{D1} and Definition \ref{D2}, respectively. Then there exist $0<\tilde{T}\leq \hat{T}$ and  positive constants $c_1,\; c_2,\, c_3$ (independent of $k$ and $u$)  such that for all $\tau \in S_{\tilde{T}}:= (0,\tilde{T})$ the following inequality holds:
	\begin{align}
\nonumber	&\| u - u_k \|_{L^2(0, 1)} + c_1|h-h_k|^2 + c_2 \int_0^{\tau} \left\Vert\frac{\partial }{\partial x}(u- u_k) \right\Vert^2 ds\\
\label{aposteriori}	& \leq  c_3\left( | h(0) - h_k(0)|^2 +  \sum_{i = 0}^{N-2} k_i^2\left\{\| R(u_k)\|_{L^2(S_{\tilde{T}}, L^2(I_i))}^2 + k_i^2 \| u_0\|_{H^2(I_i)}^2\right\}\right),
	\end{align}
	where the residual $R(u_k)$ is defined by
	\begin{align}
\label{residual}	R(u_k) := \frac{h^{\prime}_k}{h_k}  y \frac{\partial u_k}{\partial y} +  \frac{1}{h_k} {\rm Bi} \left( \frac{b(\tau)}{ m_0}  - {\rm H} u_k(\tau, 0)\right) - \frac{h^{\prime}_k}{h_k} u_k(\tau, 1) - \frac{\partial u_k}{\partial \tau}.
	\end{align}
\end{theorem}
\begin{proof}
Let $e:= u - u_k$ be the pointwise error. Using the weak formulation \eqref{a17}, we can write
\begin{align}
\nonumber\left( \frac{\partial e}{\partial \tau}, v \right)  + \frac{1}{h^2} \left( \frac{\partial e}{\partial y}, \frac{\partial v}{\partial y} \right) =&  \left[\left( \frac{\partial u}{\partial \tau}, v \right)  + \frac{1}{h^2} \left( \frac{\partial u}{\partial y}, \frac{\partial v}{\partial y} \right)\right] - \left[\left( \frac{\partial u_k}{\partial \tau}, v \right)  + \frac{1}{h^2} \left( \frac{\partial u_k}{\partial y}, \frac{\partial v}{\partial y} \right) \right]\\
\nonumber = & \frac{h^{\prime}}{h} \int_0^1 y \frac{\partial u}{\partial y} v dy +  \frac{1}{h}\text{Bi} \left( \frac{b(\tau)}{ m_0}  - \text{H} u(\tau, 0)\right)v(0) - \frac{h^{\prime}}{h} u(\tau, 1) v(1)\\
\label{aa10} &- \left[ \left( \frac{\partial u_k}{\partial \tau}, v \right)  + \frac{1}{h_k^2} \left( \frac{\partial u_k}{\partial y}, \frac{\partial v}{\partial y} \right) + \left(\frac{1}{h^2}- \frac{1}{h_k^2} \right)  \left( \frac{\partial u_k}{\partial y}, \frac{\partial v}{\partial y} \right) \right]
\end{align}
for all $v \in H^1(0, 1)$. Inserting \eqref{residual} into \eqref{aa10} yields  
\begin{align}
\nonumber \left( \frac{\partial e}{\partial \tau}, v \right)  + \frac{1}{h^2} \left( \frac{\partial e}{\partial y}, \frac{\partial v}{\partial y} \right) =&\frac{h^{\prime}}{h} \int_0^1 y \frac{\partial u}{\partial y} v dy +  \frac{1}{h} \text{Bi}\left( \frac{b(\tau)}{ m_0}  - \text{H} u(\tau, 0)\right)v(0) - \frac{h^{\prime}}{h} u(\tau, 1) v(1)\\
\nonumber& - \left(\frac{1}{h^2}- \frac{1}{h_k^2} \right)  \left( \frac{\partial u_k}{\partial y}, \frac{\partial v}{\partial y} \right) -  \frac{h^{\prime}_k}{h_k} \int_0^1 y \frac{\partial u_k}{\partial y} v dy  -  \frac{1}{h_k}\text{Bi} \left( \frac{b(\tau)}{ m_0}  - \text{H} u_k(\tau, 0)\right)v(0)\\
\label{aa0}& + \frac{h^{\prime}_k}{h_k} u_k(\tau, 1) v(1) + \left[ \int_0^1 R(u_k) v dy - \frac{1}{h_k^2} \left( \frac{\partial u_k}{\partial y},  \frac{\partial v}{\partial y}\right)\right],
\end{align} 
where $R(u_k)$
is the residual quantity defined in \eqref{residual}. Since $u_k\in V_k $, we have that $\displaystyle \frac{\partial^2 u_k}{\partial y^2} = 0$ on each $I_i := (y_i, y_{i+1})$.  The term 
\begin{align*}  \int_0^1 R(u_k) v dy - \frac{1}{h_k^2} \left( \frac{\partial u_k}{\partial y},  \frac{\partial v}{\partial y}\right) \end{align*}
becomes after integration by part
\begin{align}
\nonumber\sum_{i = 0}^{N-2}\left\{ \int_{y_i}^{y_{i+1}} R(u_k)vdy  - \frac{1}{h_k^2} \left(\frac{\partial u_k}{\partial y}(y_{i+1}) v(y_{i+1}) - \frac{\partial u_k}{\partial y}(y_{i}) v(y_{i})\right)\right\}.
\end{align}
We also get from \eqref{a19}  
\begin{align}
\label{aa1}\sum_{i = 0}^{N-2}\left\{ \int_{y_i}^{y_{i+1}} R(u_k)v_kdy  - \frac{1}{h_k^2} \left(\frac{\partial u_k}{\partial y}(y_{i+1}) v_k(y_{i+1}) - \frac{\partial u_k}{\partial y}(y_{i}) v_k(y_{i})\right)\right\} = 0 
\end{align}
for all $v_k\in V_k$.
Adding \eqref{aa1} to \eqref{aa0} while taking $v= e \in H^1(0,1)$ and $v_k = I_k e \in V_k$ gives  
 
 \begin{align}
\nonumber  \frac{1}{2} \frac{d}{d \tau}  \lVert e \rVert^2 + \frac{1}{h^2} \left\Vert \frac{\partial e}{\partial y} \right\Vert^2 &=\frac{h^{\prime}}{h} \int_0^1 y \frac{\partial u}{\partial y} e dy +  \frac{1}{h} \text{Bi}\left( \frac{b(\tau)}{ m_0}  - \text{H} u(\tau, 0)\right)e(0)\\
\nonumber & - \frac{h^{\prime}}{h} u(\tau, 1) e(1)
 - \left(\frac{1}{h^2}- \frac{1}{h_k^2} \right)  \left( \frac{\partial u_k}{\partial y}, \frac{\partial e}{\partial y} \right) -  \frac{h^{\prime}_k}{h_k} \int_0^1 y \frac{\partial u_k}{\partial y} e dy \\
 \nonumber & -  \frac{1}{h_k}\text{Bi} \left( \frac{b(\tau)}{ m_0}  - \text{H} u_k(\tau, 0)\right)e(0)
 + \frac{h^{\prime}_k}{h_k} u_k(\tau, 1) e(1) \\
 \nonumber & + \sum_{i = 0}^{N-2}\Biggl\{ \int_{y_i}^{y_{i+1}} R(u_k)(e- I_ke)dy \\
 \nonumber& - \frac{1}{h_k^2} \Biggl( \frac{\partial u_k}{\partial y}(y_{i+1})(e- I_ke) (y_{i+1}) - \frac{\partial u_k}{\partial y}(y_{i}) (e- I_ke)(y_{i})\Biggr)\Biggr\}\\
 \nonumber &= \sum_{i=1}^5I_i,
 \end{align}

where
\begin{align}
\nonumber I_1 &:=\frac{h^{\prime}}{h} \int_0^1 y \frac{\partial u}{\partial y} e dy  - \frac{h^{\prime}_k}{h_k} \int_0^1 y \frac{\partial u_k}{\partial y} e dy,\\
\nonumber I_2&:=  \frac{1}{h} \text{Bi}\left( \frac{b(\tau)}{ m_0}  - \text{H} u(\tau, 0)\right)e(0) -  \frac{1}{h_k}\text{Bi} \left( \frac{b(\tau)}{ m_0}  - \text{H} u_k(\tau, 0)\right)e(0),\\
\nonumber I_3& :=  \frac{h^{\prime}_k}{h_k} u_k(\tau, 1) e(1)  - \frac{h^{\prime}}{h} u(\tau, 1) e(1),\\
\nonumber I_4& :=  - \left(\frac{1}{h^2}- \frac{1}{h_k^2} \right)  \left( \frac{\partial u_k}{\partial y}, \frac{\partial e}{\partial y} \right), \\
\nonumber I_5&:=  \sum_{i = 0}^{N-2}\left\{ \int_{y_i}^{y_{i+1}} R(u_k)(e- I_ke)dy  - \frac{1}{h_k^2} \left(\frac{\partial u_k}{\partial y}(y_{i+1})(e- I_ke) (y_{i+1}) - \frac{\partial u_k}{\partial y}(y_{i}) (e- I_ke)(y_{i})\right)\right\}.
\end{align}
By using \eqref{energy} together with Cauchy-Schwarz and Young's inequality, we obtain 
\begin{align}
\nonumber| I_1|& \leq \frac{h^{\prime}}{h} \left\Vert \frac{\partial e}{\partial y}\right\Vert \left\Vert  e\right\Vert + \left\vert \frac{h^{\prime}}{h}  - \frac{h^{\prime}_k}{h_k}  \right\vert \| e\| \\
\label{i1}& \leq\left(\frac{\xi}{h^2} \left\Vert \frac{\partial e}{\partial y}\right\Vert^2 + c_{\xi}  \|h^{\prime}\|_{\infty}^2 \left\Vert  e\right\Vert ^2  \right) + \xi \| e\|^2 +  2c_{\xi} \left( |h - h_k|^2 + |h^{\prime} - h_k^{\prime}|^2 \right). \\
\nonumber| I_2|& \leq \text{Bi} \;\frac{b(\tau)}{m_0} \frac{1}{h h_k}| h- h_k| |e(0)| + \text{Bi}\; \text{H} \left\vert \frac{u(\tau, 0)}{h} - \frac{u_k(\tau, 0)}{h_k}\right\vert  |e(0)|\\
\nonumber&  \leq \left( \text{Bi} \frac{b^*}{m_0} \frac{1}{L^2} \hat{c} + \text{Bi} \;\text{H} \hat{c} \right)| h- h_k|  \| e\|^{1- \theta}\left\Vert \frac{\partial e}{\partial y}\right\Vert^{\theta} + c_2 \text{Bi}\; \text{H} \hat{c} \| e\|^{2(1- \theta)}\left\Vert \frac{\partial e}{\partial y}\right\Vert^{2\theta}  \\
\label{i2}&\leq \bar{\xi}| h-h_k|^2 + \xi c_{\bar{\xi}}\frac{1}{h^2}  \left\Vert \frac{\partial e}{\partial y}\right\Vert^2 + \tilde{c}^4 c_{\xi} c_{\bar{\xi}} h^2  \| e\|^{2} + \frac{\xi}{h^2} \left\Vert \frac{\partial e}{\partial y}\right\Vert^2  + \tilde{c}_1^2 c_{\xi} h^2 \| e\|^2, 
\end{align}
where
$$ \theta = \frac{1}{2},\;\; \tilde{c} := \left( \text{Bi} \frac{b^*}{m_0} \frac{1}{L^2} \hat{c} + \text{Bi}\; \text{H} \hat{c} \right)\;\;\;\;\text{and}\;\;\;\tilde{c}_1:= c_2\text{Bi}\; \text{H}\hat{c}.$$
\begin{align}
\nonumber| I_3| &\leq \left\vert \frac{h^{\prime}}{h}  - \frac{h^{\prime}_k}{h_k}  \right\vert |e(1)| + \frac{h_k^{\prime}}{h_k} |e(1)|^2\\
\label{i3}&\leq  2\bar{\xi} \left( |h - h_k|^2 + |h^{\prime} - h_k^{\prime}|^2 \right) +  \xi \frac{1}{h^2}  \left\Vert \frac{\partial e}{\partial y} \right\Vert^2 + c_{\bar{\xi}} c_{\xi} c_3^4 \hat{c}^4 \| e\|^{2}h^2 + c \left(\frac{\xi} {h^2}  \left\Vert \frac{\partial e}{\partial y} \right\Vert^2 + c_{\xi} \| e\|^2 \right).\\
\nonumber| I_4| & \leq |h - h_k| \frac{h+ h_k}{h^2 h_k^2} \left\Vert \frac{\partial u_k}{\partial y} \right\Vert\left\Vert \frac{\partial e}{\partial y}\right\Vert\\
\label{i4}&\leq \xi | h-h_k|^2 + c_{\xi} c^2(h_0, L)\frac{1}{h^2} \left\Vert \frac{\partial e}{\partial y}\right\Vert^2.
\end{align}
To bound $|I_5|$ from above, we use the fact that $ I_k e$ is the Lagrange interpolant of $e$ with the property $(e- I_k e)(y_i) = 0,\; i \in \{0,1,2, \cdots, N-1\}.$ We have 
\begin{align}
\nonumber | I_5| &\leq  \sum_{i = 0}^{N-2}\int_{y_i}^{y_{i+1}} R(u_k)(e- I_ke)dy\\
\nonumber&\leq \sum_{i = 0}^{N-2} \| R(u_k)\|_{L^2(I_i)} \| e- I_ke \|_{L^2(I_i)}\\
\nonumber& \leq \tilde{c}  \sum_{i = 0}^{N-2} \| R(u_k)\|_{L^2(I_i)} k_i \left\Vert\frac{\partial e}{\partial y} \right\Vert_{L^2(I_i)}\\
\nonumber& \leq \tilde{c} \left(  \sum_{i = 0}^{N-2} \| R(u_k)\|_{L^2(I_i)}^2 k_i^2\right)^{\frac{1}{2}} \left(  \sum_{i = 0}^{N-2} \left\Vert\frac{\partial e}{\partial y} \right\Vert_{L^2(I_i)}^2\right)^{\frac{1}{2}}\\
\nonumber& =  \tilde{c} \left(  \sum_{i = 0}^{N-2} \| R(u_k)\|_{L^2(I_i)}^2 k_i^2\right)^{\frac{1}{2}} \left\Vert\frac{\partial e}{\partial y} \right\Vert_{L^2(0, 1)}.
\end{align}
By using Young's inequality, we obtain
\begin{align}
\label{i5}|I_5| \leq \frac{\xi}{h^2} \left\Vert\frac{\partial e}{\partial y}\right\Vert^2 + c_{\xi}\tilde{c}^2 h^2 \sum_{i = 0}^{N-2} \| R(u_k)\|_{L^2(I_i)}^2 k_i^2.
\end{align}
 It follows from \eqref{i1}--\eqref{i5} that for all  $\xi,\; \bar{\xi}>0$,  there exist positive constants $K_1,\; K_2,\; K_3$ and $K_4$  such that
\begin{align}
\nonumber \frac{1}{2} \frac{d}{d \tau}  \lVert e \rVert^2 + \frac{1}{h^2} \left\Vert \frac{\partial e}{\partial y} \right\Vert^2 &\leq K_1\| e\|^2 + K_2 | h-h_k|^2 \\
 \nonumber &+  \frac{1}{h^2} K_3 \left\Vert \frac{\partial e}{\partial y} \right\Vert^2 +  K_4 \sum_{i = 0}^{N-2} \| R(u_k)\|_{L^2(I_i)}^2 k_i^2.
 \end{align}
 Let $\delta>0$ be a fixed, sufficiently small. Adding $\frac{\delta}{2} \frac{d}{d\tau} |h - h_k|^2$ on both sides and using \eqref{a} yields
 \begin{align}
 \nonumber \frac{1}{2} \frac{d}{d \tau} \left( \lVert e \rVert^2  + \delta | h-h_k|^2\right)+ \frac{1}{L^2}(1 - K_3 - 3\delta A_0^2\eta) \left\Vert \frac{\partial e}{\partial y} \right\Vert^2 &\leq K_1 \|e\|^2 + K_2 | h-h_k|^2 + 3\delta A_0^2c_{\eta}\| e\|^2\\
\nonumber  &+ C\delta\| h - h_k\|^2+  K_4\sum_{i = 0}^{N-2} \| R(u_k)\|_{L^2(I_i)}^2 k_i^2.
 \end{align}
 We choose $\xi>0,\; \bar{\xi}>0$ and  $\eta>0$ in such a way that $1 - K_3 - 3\delta A_0^2\eta \geq 0$. Then it exists $K_5>0$ such that\\
  \begin{align}
 \nonumber \frac{1}{2} \frac{d}{d \tau} \left( \lVert e \rVert^2  + \delta | h-h_k|^2\right)+ \frac{1}{L^2}(1 - K_3 - 3\delta A_0^2\eta) \left\Vert \frac{\partial e}{\partial y} \right\Vert^2 &\leq K_5( \|e\|^2 + \delta | h-h_k|^2 )+ \\
\label{gronwall} &+ K_4 \sum_{i = 0}^{N-2} \| R(u_k)\|_{L^2(I_i)}^2 k_i^2.
 \end{align}
  Applying Gronwall's inequality to \eqref{gronwall} for the quantity  $\lVert e \rVert^2  + \delta | h-h_k|^2$  and  using the initial condition  $$\| e(0)\|_{L^2(0, 1)}^2 = \sum_{i=0}^{N-2}\| e(0)\|_{L^2(I_i)}^2\leq k_i^4 \| u_0\|_{H^2(I_i)}^2, $$
 it exists a constant $c(\tilde{T}, L)$ such that
 \begin{align}
 \label{aa11}\lVert e \rVert^2  + \delta | h(\tau)-h_k(\tau)|^2  \leq c(\tilde{T}, L) \left( |h(0) - h_k(0)|^2+ k_i^4 \| u_0\|_{H^2(I_i)}^2  +  \sum_{i = 0}^{N-2} \int_{0}^{\tau} \| R(u_k)\|_{L^2(I_i)}^2 k_i^2 ds\right). 
 \end{align}
 By integrating \eqref{gronwall} on $(0, \tau)$ and by using \eqref{aa11}, it exists another constant $c(\tilde{T}, L)>0$ such that the following inequality holds: 
 \begin{align*}
 \int_0^{\tau} \left\Vert\frac{\partial }{\partial x}(u- u_k) \right\Vert^2 ds \leq c(\tilde{T}, L) \left( |h(0) - h_k(0)|^2+ k_i^4 \| u_0\|_{H^2(I_i)}^2  +  \sum_{i = 0}^{N-2} \int_{0}^{\tau} \| R(u_k)\|_{L^2(I_i)}^2 k_i^2 ds\right). 
 \end{align*}
 This concludes the proof of Theorem \ref{post}.
 \end{proof}
  %The estimate on the  time integral on the concentration  gradient arising in \eqref{aposteriori} is obtained by integrating \eqref{gronwall} on $(0, \tau)$.
  
  \section{Numerical illustrations} \label{numerical}
  
  In this section, we firstly present our simulation results for both the dense and foam rubber. The difference in the two cases is incorporated in the choice of parameters. To approximate numerically the weak solution to  \eqref{a19}--\eqref{a22},  we use the method of lines; for more details see, for instance, \cite{larsson2008partial}. Firstly, the model equations are discretized in space by means of the finite element method. The resulting time-dependent system of ordinary differential equations  is tackled  via the solver \texttt{odeint} in Python; see \cite{linge2020programming} for details on Python and  \cite{johansson2018numerical} for details on the solver. We refer the reader to see our previous work \cite{nepal2021moving} for the laboratory  experiments, numerical method and simulation results where we investigated the parameter space by exploring eventual effects of the choice of parameters on the overall diffusants penetration process.\\
 We take as observation time $T_f = 40$ minutes for the final time with time step $\Delta t = 1/1000$ minutes. We choose the number of space discretization points $N$ to be $100$. The values of parameters are taken to be $s_0 = 0.01$ (mm), $m_0$ = 0.1 (gram/mm$^3$) and  $b=1$ (gram/mm$^3$). We take the value $3.66 \times 10^{-4}$ (mm$^2$/min) for the diffusion constant $D$ \cite{morton2013rubber}, $0.564$ (mm/min) for absorption rate $\beta$ \cite{rezk2018determination} and $2.5$ for Henry's constant \text{H} \cite{bohm1998moving}. For the dense rubber, we choose $\sigma(s(t)) = s(t)/10$ (gram/mm$^3$) and $a_0 = 500$ (mm$^4$/sec/gram) while we choose $\sigma(s(t)) = s(t)/50$ (gram/mm$^3$) and $a_0 = 2000$ (mm$^4$/sec/gram) for the foam rubber case.\\
  \begin{figure}[ht]
 	\centering
 	\includegraphics[width=0.43\textwidth]{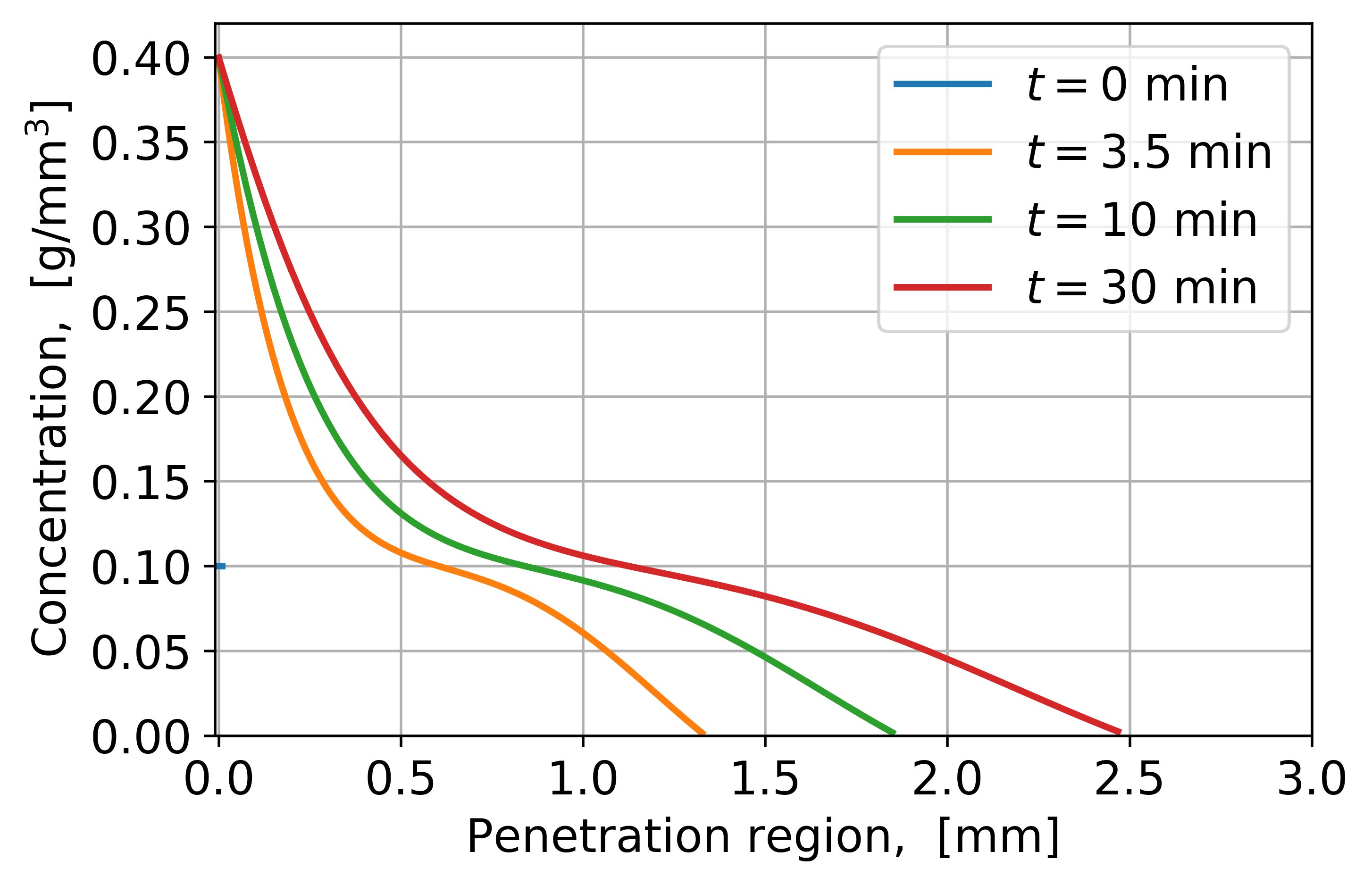}
 	\hspace{0.1cm}
 	\includegraphics[width=0.40\textwidth]{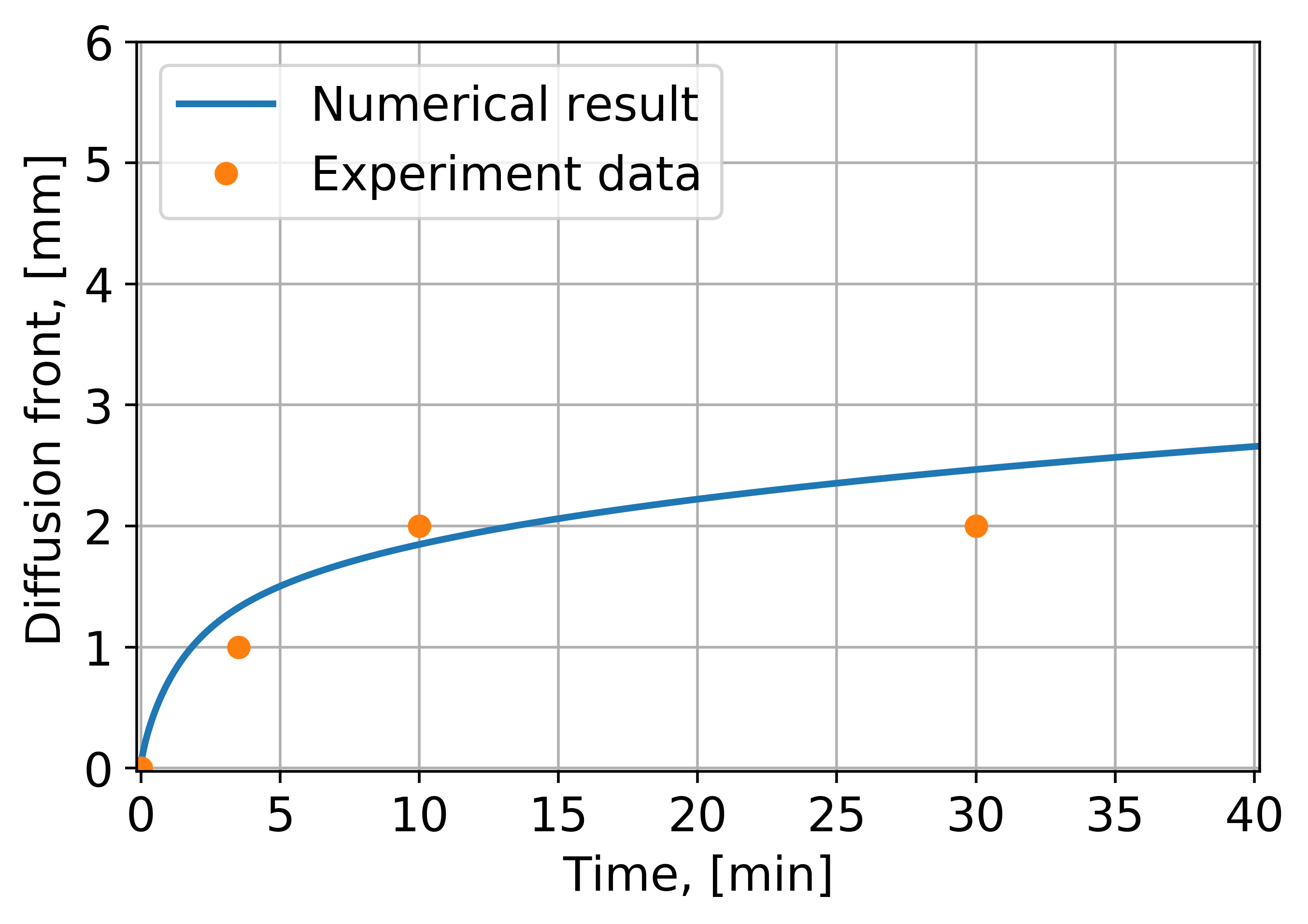}
 	\caption{Dense rubber case. Left: Concentration profile of diffusant. Right: Position of the moving boundary.}
 	\label{Fig:24}
 \end{figure}
 
  \begin{figure}[ht]
 	\centering
 	\includegraphics[width=0.43\textwidth]{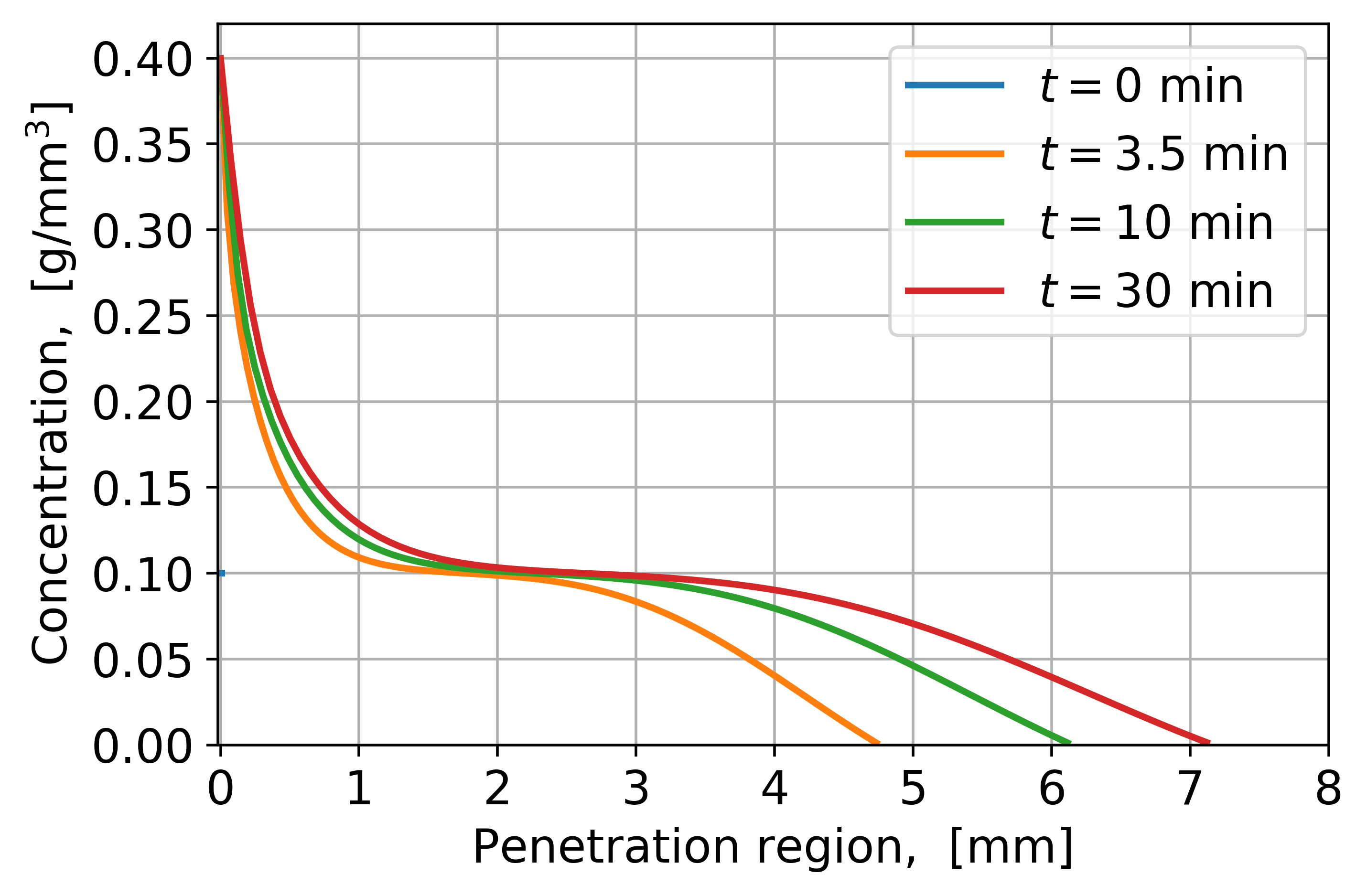}
 	\hspace{0.1cm}
 	\includegraphics[width=0.40\textwidth]{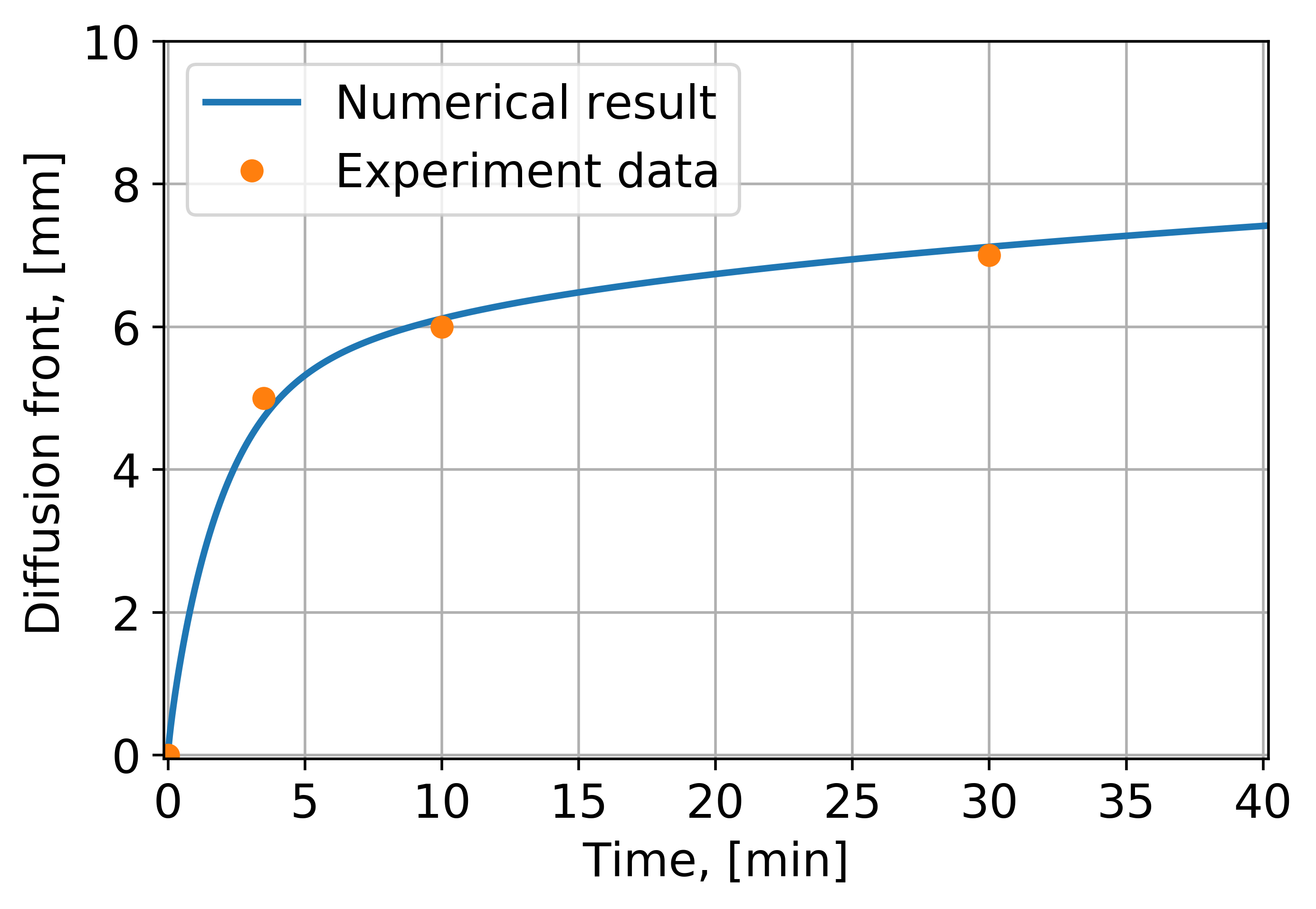}
 	\caption{Foam rubber case. Left: Concentration profile of diffusants. Right:  Position of the moving boundary.}
 	\label{Fig:25}
 \end{figure}
In Figure \ref{Fig:24} and Figure \ref{Fig:25} we show the concentration profile of the penetrating diffusant, and respectively, the position of the moving boundary for the dense rubber and foam rubber respectively.  
Comparing the diffusant concentration profile in Figure \ref{Fig:24} and Figure \ref{Fig:25},  we notice in both cases that, within a short time of release of diffusant from its initial
position, the diffusant quickly enters the rubber from the left boundary and then starts diffusing inside  displacing a penetration front. In bothe Figure \ref{Fig:24} and Figure \ref{Fig:25}, we compare simulation results against experimental data for the position of moving boundary. Both plots  show a good agreement between model and experiment. % for the position of the diffusant front for the dense rubber and foam rubber respectively. 

  Finally, we wish to point out that  the order of convergence of our FEM scheme is consistent with the estimates stated in  \eqref{aprioriestimate}.
As we are not aware of an exact solution to \eqref{a19}--\eqref{a22},  {\clb we compute the finite element approximation of our weak solution on a fine mesh  (say, with $640$ nodes) and denote it by $u_{\tilde{k}}$. We use this  $u_{\tilde{k}}$ as the reference solution for computing the errors and convergence orders. % Let $u_{k, \ell}^n$ be approximation solution of $u_k(\tau_n, y_{\ell})$ where $\tau_n = n \Delta \tau,\; n\in \mathbb{N}$ and $\ell \in \{0, 1, \cdots, N-1 \}$. 
%Let the pointwise error be defined as $e_{\ell}^j = \tilde{u}_\ell^j- u_{\ell}^j$. 
We make use of the discrete $\ell^2(Q(\hat{T}))$ norm which we denote here as
 %\begin{align*}
 %  \label{discrete}  \left\Vert e \right\Vert_{L^2(S_{\hat{T}}, L^2(0, 1))} := \left(\Delta \tau \sum_{j=0}^{N_t}  \sum_{l=0}^{N_i} k_i|e|^2 \right)^{\frac{1}{2}}. 
 %\end{align*}
 \begin{align*}
   \label{discretenorm} e(k_i):= \left\Vert u_{\tilde{k}}(\tau, y) - u_{k_i}(\tau, y) \right\Vert_{L^2(S_{\hat{T}}, L^2(0, 1))} = \left(\Delta \tau k_i \sum_{j=0}^{N_t}  \sum_{\ell=0}^{N-1} |u_{\tilde{k}}(\tau_j, y_{\ell}) - u_{k_i}(\tau_j, y_{\ell}) |^2 \right)^{\frac{1}{2}}. 
 \end{align*}
%where $e_i^j$ is the error associate with $k_i$ and $\Delta t$. 
Here $\Delta \tau$ is the uniform size of the $N_t+1$ time steps, while $\{k_1, k_2, k_3,\cdots\}$ with $k_i >k_{i+1}$ for $i\in\{1,2,\cdots\}$ is a finite collection of the different mesh sizes used in the computations.}

We  determine the convergence order based on any two consecutive calculations of discrete errors using two different mesh sizes.
To this end, we perform the computations on a sequence of grids with mesh size $k$ that are halved in each step. Thus, we use the following formula  to compute the convergence order $r$:

\begin{align*}
    r := \log_2\left(\frac{e(k_i)}{e(k_{i+1})}\right).
\end{align*}

 \begin{figure}[ht] 
 	\centering
 	\includegraphics[width=0.43\textwidth]{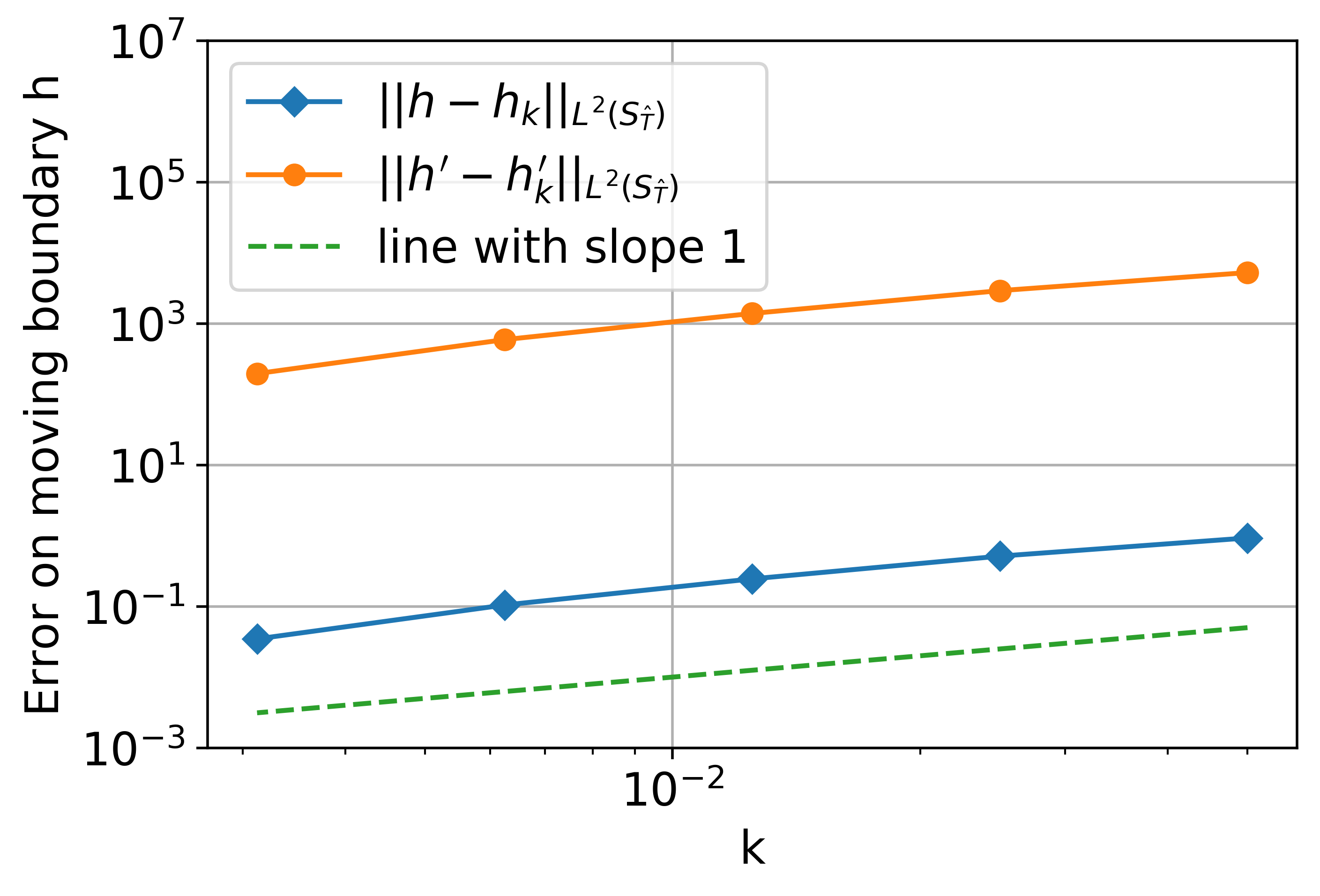}
 	\hspace{0.1cm}
 	\includegraphics[width=0.43\textwidth]{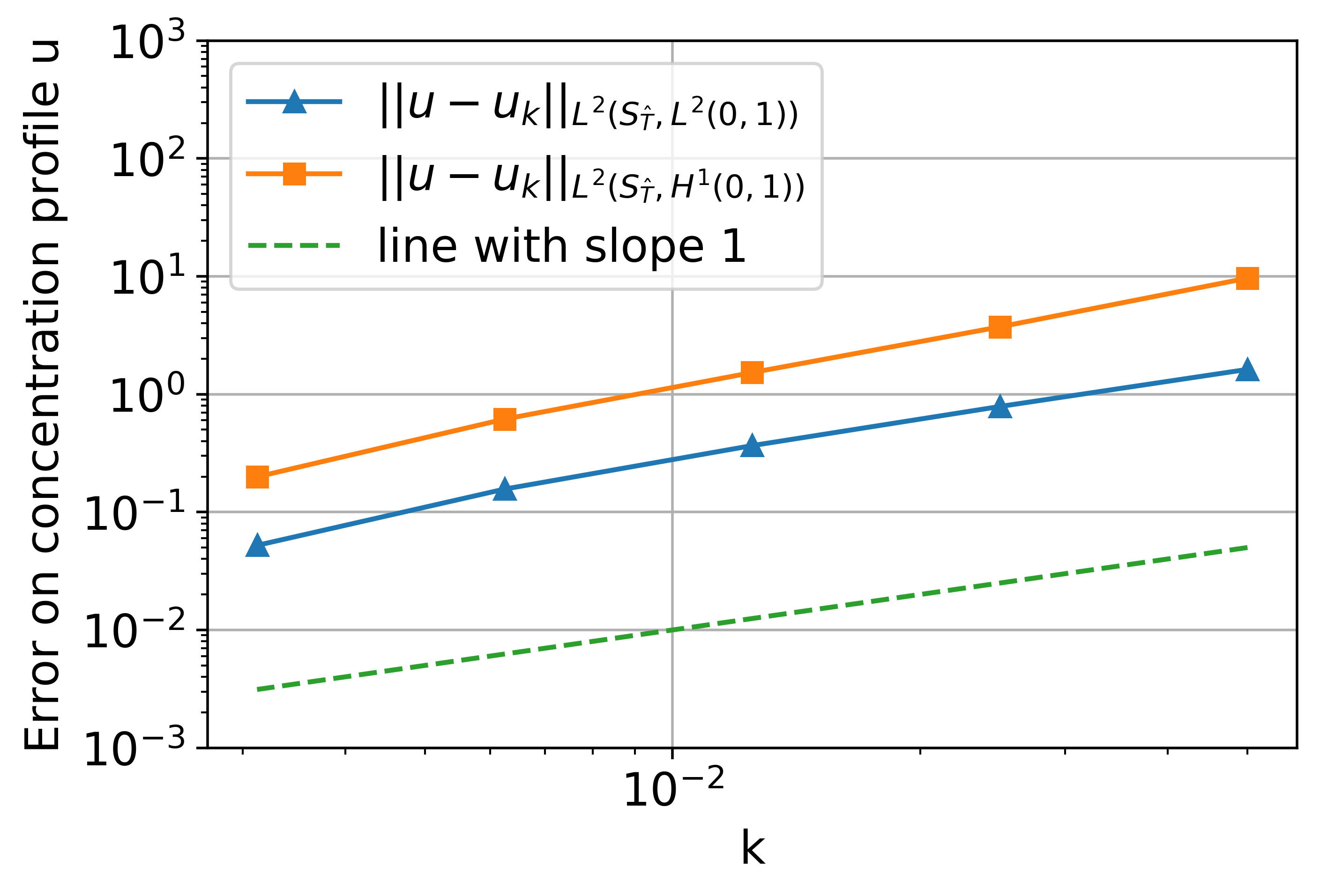}
 	\caption{ Convergence order when time step size $\Delta t = 10^{-4}$  is fixed. Dash lines are lines of slope 1.  Left: Log log scale plot of error on the boundary $\|h - h_k\|_{L^2(S_{\hat{T}})}$ (circles) and $\|h^{\prime} - h_k^{\prime}\|_{L^2(S_{\hat{T}})}$ (diamonds).  Right: Log log scale plot of error on the concentration $\|u - u_k\|_{L^2(S_{\hat{T}}, L^2(0, 1))}$ (triangles) and $\|u - u_k\|_{L^2(S_{\hat{T}}, H^1(0, 1))}$ (squares).}
 	\label{PlotlogLog}
 \end{figure}
We show in Figure \ref{PlotlogLog}  the computed convergence order for the approximation  of the moving boundary position and of the concentration profile. This is done in various norms for $N= 20, 40, 80, 160$, and  $320$. These numerical results are in agreement with the convergence order proven in Section \ref{mainresults}.

\section{Conclusion}\label{conclusion}

The goal of this work was to analyze  the errors produced by a semi-discrete finite element approximation of the weak solution of  moving boundary problem modeling the penetration of diffusants into rubber. We obtained  the  \textit{a priori}  error estimate  \eqref{aprioriestimate} for the diffusant concentration profile as well as for the position and speed of the moving boundary. The convergence rate is of order of $\mathcal{O}(1)$ -- the deviation from optimality is due to the nonlinear coupling produced by the presence of the unknown moving boundary. %Based on it, we observed that the  finite element approximation to  weak solution of the semi-discrete problem  converges to the  weak solution of the continuous problem when the space step size goes to zero and the convergence rate is equal to 1 for the concentration profile, and respectively, for the position of the moving boundary. 
Additionally, we obtained the \textit{a posteriori} error \eqref{aposteriori}. Finally, we illustrated numerically the basic output of our model. It turns out that results are in the expected experimental range and they can be obtained in practice using convergence rates closed to the theoretical ones. %provided the numerical simulation for our problem and compared  theoretically obtained convergence rates with numerically obtained convergence rates for the concentration profile, and respectively, for the position of the moving boundary.

\section*{Acknowledgements} The authors acknowledge fruitful discussions with U. Giese, N. Kr\"{o}ger, R. Meyer (Deutsches Institut f\"{u}r Kautschuktechnologie, Hannover, Germany), T. Aiki (Japan Women's University, Tokyo, Japan), and K. Kumazaki (Nagasaki University, Japan) about the modeling, mathematical analysis, and simulation of rubber-based materials exposed to environmental conditions. The work of S.N.  and A.M. is financed partly by the Swedish Research Council's project "{\em  Homogenization and dimension reduction of thin heterogeneous layers}", grant nr. VR 2018-03648. A.M. also thanks the Knowledge Foundation for the grant KK 2019-0213, which led to the formulation of this problem setting. 
\begin{center}
	\bibliographystyle{plain}
	\bibliography{NumAnalysis_Polymer}
\end{center} 

\end{document}